\pdfoutput=1
\RequirePackage{ifpdf}
\ifpdf 
\documentclass[pdftex]{sigma}
\else
\documentclass{sigma}
\fi

\usepackage{euscript,mathrsfs,enumerate}

\def\R{\mathbb{R}}
\def\C{\mathbb{C}}

\def\ve{\mathsf{e}}
\def\vv{\mathsf{v}}
\def\vx{\boldsymbol{\mathrm{x}}}

\def\vomega{\boldsymbol{\omega}}
\def\vOmega{\boldsymbol{\Omega}}

\def\vy{\mathsf{y}}
\def\vw{\mathsf{w}}
\def\vu{\mathsf{u}}
\def\mudown{\check{\mu}}

\def\vbeta{\boldsymbol{\beta}}
\def\vtheta{\boldsymbol{\theta}}
\def\valpha{\boldsymbol{\alpha}}
\def\vOmega{\boldsymbol{\Omega}}

\def\vX{\boldsymbol{\mathrm{X}}}
\def\ba{\boldsymbol{\mathrm{a}}}
\def\di{\partial}
\def\ri{\mathrm{i}}

\renewcommand{\Im}{\operatorname{Im}}
\renewcommand{\Re}{\operatorname{Re}}
\def\adj{\operatorname{adj}}
\def\II{\operatorname{II}} 
\def\calS{\mathcal{S}}

\def\calT{\mathcal{T}}
\def\F{\mathfrak{F}}
\def\I{\mathcal{I}}
\def\calI{\mathcal{I}}
\def\J{\EuScript{J}}
\def\V{\mathscr{V}}
\def\W{\mathscr{W}}
\def\D{\mathcal{D}}
\def\w{\omega}
\def\tw{\tilde\omega}
\def\wt{\tilde\omega}

\def\ghat{\widehat{g}}
\def\hatnabla{\widehat{\nabla}}

\def\tr{\operatorname{tr}}
\def\Mhat{\widehat{M}}

\newcommand{\intprod}{\mathbin{\raisebox{.4ex}{\hbox{\vrule height .5pt width 5pt depth 0pt \vrule height 3pt width .5pt depth 0pt}}}}

\newtheorem{Theorem}{Theorem}[section]
\newtheorem{Lemma}[Theorem]{Lemma}
\newtheorem{Proposition}[Theorem]{Proposition}

 {\theoremstyle{definition}
\newtheorem{Definition}[Theorem]{Definition}
\newtheorem{Example}[Theorem]{Example}
\newtheorem{Remark}[Theorem]{Remark}}

\numberwithin{equation}{section}

\begin{document}

\newcommand{\arXivNumber}{2006.15119}

\renewcommand{\PaperNumber}{023}

\FirstPageHeading

\ShortArticleName{Twisted-Austere Submanifolds in Euclidean Space}

\ArticleName{Twisted-Austere Submanifolds in Euclidean Space}

\Author{Thomas A.~IVEY~$^{\rm a}$ and Spiro KARIGIANNIS~$^{\rm b}$}

\AuthorNameForHeading{T.A.~Ivey and S.~Karigiannis}

\Address{$^{\rm a)}$~Department of Mathematics, College of Charleston, USA}
\EmailD{\href{mailto:iveyt@cofc.edu}{iveyt@cofc.edu}}
\URLaddressD{\url{http://iveyt.people.cofc.edu}}

\Address{$^{\rm b)}$~Department of Pure Mathematics, University of Waterloo, Canada}
\EmailD{\href{mailto:karigiannis@uwaterloo.ca}{karigiannis@uwaterloo.ca}}
\URLaddressD{\url{http://www.math.uwaterloo.ca/~karigiannis}}

\ArticleDates{Received October 13, 2020, in final form March 02, 2021; Published online March 10, 2021}

\Abstract{A twisted-austere $k$-fold $(M, \mu)$ in ${\mathbb R}^n$ consists of a $k$-dimensional sub\-mani\-fold~$M$ of ${\mathbb R}^n$ together with a closed $1$-form $\mu$ on~$M$, such that the second fundamental form~$A$ of~$M$ and the $1$-form $\mu$ satisfy a particular system of coupled nonlinear second order PDE. Given such an object, the ``twisted conormal bundle'' $N^* M + \mu$ is a special Lagrangian submanifold of ${\mathbb C}^n$. We~review the twisted-austere condition and give an explicit example. Then we focus on twisted-austere 3-folds. We~give a geometric description of all solutions when the ``base''~$M$ is a cylinder, and when~$M$ is austere. Finally, we~prove that, other than the case of a generalized helicoid in~${\mathbb R}^5$ discovered by Bryant, there are no other possibilities for the base~$M$. This gives a complete classification of twisted-austere $3$-folds in~${\mathbb R}^n$.}


\Classification{53B25; 53C38; 53C40; 53D12; 58A15}

\section{Introduction}

Special Lagrangian submanifolds are a special class of $n$-dimensional submanifold in $\C^n$, and more generally in Calabi--Yau $n$-folds. They were introduced by Harvey--Lawson~\cite{HL} and were the first modern example of \emph{calibrated} submanifolds. They are a class of \emph{minimal} (vanishing mean curvature) submanifolds characterized by a first order nonlinear PDE, and in fact are absolutely locally volume minimizing in their homology class. Special Lagrangian submanifolds also play a key role in \emph{mirror symmetry} through the Strominger--Yau--Zaslow conjecture~\cite{SYZ}. They have been extensively studied by many authors. An excellent reference summarizing much of the work on special Lagrangian geometry up to the time of its publication is the textbook~\cite{J-book} of Joyce on calibrated geometry.

One particular construction of special Lagrangian submanifolds in $\C^n$ first appeared in~\cite{HL} and is known as the \emph{conormal bundle construction}. Given a $k$-dimensional submanifold $M$ of~$\R^n$, Harvey--Lawson showed that its conormal bundle $N^*M$ is special Lagrangian in~\mbox{$T^* \R^n = \C^n$} if~and only if $M$ is \emph{austere}, which means that all the odd degree elementary symmetric polynomials in the eigenvalues of the second fundamental form vanish. Note that this is in general a~\emph{fully nonlinear} second order PDE on the immersion of $M$ in $\R^n$. The~conormal bundle construction was later reviewed in detail, and generalized to the exceptional holonomies $\mathrm{G}_2$ and $\mathrm{Spin}(7)$, by Ionel--Karigiannis--Min-Oo in~\cite{IKM}. Austere submanifolds in Euclidean space have been studied by several authors, including Bryant~\cite{Br} and Ionel--Ivey~\cite{II,II2}.

A generalization of the conormal bundle construction was introduced by Borisenko~\cite{Bo} and later significantly extended by Karigiannis--Leung~\cite{KL}. The~idea is as follows. Let $M$ be a~$k$-dimen\-sional submanifold of $\R^n$. Then $T^* \R^n |_M = N^* M \oplus T^* M$. Let $\mu$ be a \emph{closed} $1$-form on~$M$. Define the ``twisted conormal bundle'' to be the $n$-dimensional submanifold $N^* M + \mu = \{ (\nu_p, \mu_p) \,|\, \nu_p \in N^*_p M \}$ of $T^* \R^n = \C^n$. This is the total space of an affine bundle over $M$ whose fibres are affine translates of the conormal spaces, translated by the $1$-form $\mu$. In~\cite{KL} it was proved that $N^* M+\mu$ is special Lagrangian if and only if the second fundamental form $A$ of $M$ in $\R^n$ and the $1$-form $\mu$ satisfy a system of coupled second order fully nonlinear PDE, which we call the \emph{twisted austere} equations. This result is stated explicitly in Theorem~\ref{theoremzero}. (Borisenko had only considered the case when $\mu$ is exact, $n=3$, and $k=2$.)

Both the original construction of Harvey--Lawson and the ``twisted version'' of Borisenko and Karigiannis--Leung produce examples of \emph{ruled} special Lagrangian submanifolds. Joyce~\cite{J-ruled} has also studied ruled special Lagrangian submanifolds in $\C^n$.

We consider the case of \emph{twisted-austere pairs} $\big(M^k, \mu\big)$ in $\R^n$ for $k=1,2,3$ and any $n$. The~cases $k=1,2$ are trivial to classify completely. The~case $k=3$ is significantly more involved. We obtain a complete classification and give a geometric description of all possibilities.

\textbf{Organization of the paper and summary of results.} In Section~\ref{sec:preliminaries} we review the twisted-austere condition, and completely describe the cases $k=1$ and $k=2$, as well as the case when~$M$ is totally geodesic. We~also present an explicit nontrivial solution when $k=2$ and $n=3$, giving a special Lagrangian submanifold in $\C^3$. The~remainder of the paper is concerned with the non totally geodesic case when $k=3$.

Section~\ref{sec:tw3folds} establishes some general results on twisted-austere pairs $\big(M^3, \mu\big)$ where $M$ is not totally geodesic. The~main result is Theorem~\ref{theoremone}, where we show that $M$ is either a generalized helicoid swept out by planes in $\R^5$, or else $n$ is arbitrary and $M$ is ruled by lines. Section~\ref{cylindersec} is concerned with the particular case when~$M$ is a cylinder. The~main result is Theorem~\ref{cylinderthm}, where we give a geometric characterization of this case, in terms of a minimal surface $\Sigma$ in $\R^{n+1}$ and a closed $1$-form $\lambda$ on $\Sigma$ with prescribed codifferential.

Section~\ref{austeresec} is the heart of the paper, where we comprehensively study the case in which the base $M$ is austere. This study breaks up naturally into two cases, called the ``split case'' and the ``non-split case'', characterized by algebraic properties on the covariant derivative $\nabla \mu$. Each case then breaks up into subcases. In the split case, $M$ can be either a cylinder, a cone, or~a~``twisted cone''. The~first two subcases also occur in the non-split case. The~two cylinder subcases are related to the results of Section~\ref{cylindersec}. In all subcases the twisted-austere pairs $\big(M^3, \mu\big)$ with austere base $M$ are related to geometric data on a surface $\Sigma$, being the cross-section of the cylinder or~the link of the (twisted) cone. Using this data, the $1$-form $\mu$ is described explicitly.

Finally in Section~\ref{sec:classification} we outline the proof of our classification, which is Theorem~\ref{thm:class}. We~prove that the pairs $\big(M^3, \mu\big)$ studied in the earlier sections are the only possibilities. Three appendices follow, collecting various technical results that are used in the main body of the paper.

\section{Preliminaries} \label{sec:preliminaries}

In this section we review the \emph{twisted-austere} condition for a pair $\big(M^k, \mu\big)$ where $M^k$ is a $k$-dimensional submanifold of $\R^n$ and $\mu$ is a smooth $1$-form on $M$. We~also discuss the cases $k=1$ and $k=2$ in detail, as well as the case when $M^k \subset \R^n$ is totally geodesic. The remainder of the paper is concerned with the case $k=3$ for $M^k$ not totally geodesic.

\begin{Definition} 
Let $M^k$ be a $k$-dimensional submanifold of $\R^n$ and a let $\mu$ be a smooth $1$-form on $M$. Define $L = N^* M + \mu$ to be the $n$-dimensional submanifold of $T^* \R^n$ given by
\begin{gather} \label{eq:twisted-austere}
N^* M + \mu = \big\{(x, \xi + \mu_x) \in T^* \R^n|_M \,|\, x \in M,\, \xi \in N^*_x M \big\}.
\end{gather}
We say that $(M,\mu)$ is a \emph{twisted-austere pair} if $L = N^*M+ \mu$ is a special Lagrangian submanifold inside $T^* \R^n$ with respect to some phase. Following~\cite{KL}, we~refer to this as the \emph{Borisenko construction}.
\end{Definition}

It is shown in~\cite{KL} that $L$ is Lagrangian if and only if $\nabla \mu$ is a symmetric tensor on $M$, that is ${\rm d} \mu = 0$. The~conditions under which $L$ is special Lagrangian are more involved. In what follows, let $A^\nu = \nu \cdot \II$ denote the second fundamental form of $M$ in the normal direction $\nu$, and let $B = \nabla \mu$. We~use the same letters to denote the matrices that represent these covariant tensors with respect to a local orthonormal frame field $\ve_1, \dots, \ve_k$ on $M$. (For example, $B_{ij} = B(\ve_i, \ve_j)$, and the Lagrangian condition
is equivalent to $B$ being a symmetric matrix.)

\begin{Theorem}[Karigiannis--Leung~\cite{KL}] \label{theoremzero}
Fix a \emph{phase angle} $\theta \in [0, 2 \pi)$. Let $C=I + \ri B$, and define the \emph{cophase angle} $\phi$ by
\begin{gather} \label{angleeq}
\phi = \theta - (n-k) \dfrac{\pi}2.
\end{gather}
Then $(M, \mu)$ is a twisted-austere pair with phase ${\rm e}^{\ri \theta}$ if and only if the following three conditions all hold:
\begin{gather} \label{eq:tw1}
{\rm d}\mu = 0,
\\ \label{eq:tw2}
\Im \big({\rm e}^{\ri \phi} \det C \big) = 0,
\\ \label{eq:tw3}
\Im \big( \ri^j \sigma_j \big(A^\nu C^{-1} \big) \big) = 0, \qquad \text{for all $\nu$ and all $j=1, \dots, k$}.
\end{gather}
Here $\sigma_j$ denotes the $j^{\text{th}}$ elementary symmetric function of the eigenvalues of a matrix, so in particular $\sigma_1 = \tr$ and $\sigma_k = \det$. $($See Appendix~$\ref{sec:sigma-adj}$ for more details.$)$
\end{Theorem}

\begin{Remark}
In~\cite[Theorem 2.3]{KL} the definition of $\phi$ is the negative of what we have in~\eqref{angleeq}, because in~\cite{KL} the definition of special Lagrangian with phase ${\rm e}^{\ri \theta}$ meant calibrated with respect to ${\rm e}^{-\ri \theta} {\rm d}z^1 \wedge \cdots \wedge {\rm d}z^n$, whereas we take it to mean calibrated with respect to ${\rm e}^{\ri \theta} {\rm d}z^1 \wedge \cdots \wedge {\rm d}z^n$, which is standard.
\end{Remark}

Note that condition~\eqref{eq:tw3} is really a sequence of conditions for each normal direction $\nu$, as~follows:
\begin{gather} \label{eq:seq}
\Re \big(\sigma_1 \big(A^\nu C^{-1}\big) \big) = 0, \qquad
\Im \big(\sigma_2 \big(A^\nu C^{-1}\big) \big) = 0, \qquad
\Re \big(\sigma_3 \big(A^\nu C^{-1}\big) \big) = 0, \qquad \dots.
\end{gather}
It is useful to rewrite equation~\eqref{eq:tw3} in the extreme cases $j=1$ and $j=k$ as follows. By the linearity of $\sigma_1 = \tr$, we~have
\begin{gather*}
2 \Re \big(\sigma_1 \big(A^\nu C^{-1}\big)\big) = \tr \big(A^{\nu} \big((I + \ri B)^{-1} + (I - \ri B)^{-1}\big)\big).
\end{gather*}
But because we can diagonalize the symmetric matrix $B$, it is easy to see that $(I + \ri B)^{-1} +$ $(I - \ri B)^{-1} = 2 \big(I + B^2\big)^{-1}$. Thus we find that
\begin{gather*}
2 \Re \big( \sigma_1 \big(A^\nu C^{-1}\big) \big) = 2 \tr \big( A^{\nu} \big(I + B^2\big)^{-1} \big).
\end{gather*}
Hence, the condition~\eqref{eq:tw3} in the $j=1$ case can be rewritten as
\begin{gather} \label{eq:tw3-special1}
\tr \big(A^{\nu} \big(I + B^2\big)^{-1}\big) = 0 \qquad \text{for all $\nu$}.
\end{gather}
Because $\sigma_k = \det$ is multiplicative, we~have $\sigma_k \big(A^{\nu} C^{-1}\big) = \det A^{\nu} \det C^{-1}$. Hence, the condition~\eqref{eq:tw3} in the $j=k$ case can be rewritten as
\begin{gather}\label{eq:tw3-specialk}
(\det A^{\nu}) \Im \left(\frac{\ri^k}{\det C}\right) = 0 \qquad \text{for all $\nu$}.
\end{gather}

The simplest case of the twisted austere condition is when $M^k \subset \R^n$ is totally geodesic.

\begin{Proposition} \label{prop:totally-geodesic}
Suppose that $M^k \subset \R^n$ is \emph{totally geodesic} and complete. Without loss of~gene\-ra\-lity we can take $M^k = \R^k \subset \R^n$. Then the Borisenko construction yields a pro\-duct $K \times \R^{n-k}$, where $K \subset T^* \R^k$ is a special Lagrangian submanifold which is the graph of~$\mu$.
\end{Proposition}
\begin{proof}
Since $\R^k$ is totally geodesic, we~have $A^{\nu} = 0$ for all $\nu$. Thus the sequence of conditions~\eqref{eq:seq} are trivially satisfied. We~have $N^* M = \R^k \times \R^{n-k} \subset \R^n \times \R^n = T^* \R^n$. The~closed $1$-form $\mu$ is necessarily exact, so $\mu = {\rm d} f$ for some $f \in C^{\infty} \big(\R^k\big)$. Equation~\eqref{eq:tw2} becomes $\Im \big( {\rm e}^{\ri \phi} \det (I + {\rm i} \operatorname{Hess} f) \big) = 0$ for $\phi = \theta - (n-k) \frac{\pi}{2}$. Then by~\cite[Theorem~2.3]{HL}, the graph of~$\mu$ in~$T^* \R^k$ is a~special Lagrangian submanifold $K$ of $T^* \R^n$ with phase ${\rm e}^{\ri \theta}$. (See~\cite[Theorem 2.3]{KL} for~discus\-sion about the phase.) Hence $L = N^* \R^k + \mu = K \times \R^{n-k}$ as claimed.
\end{proof}

A discussion of the cases $k=1,2$ of the twisted-austere condition was given in~\cite[Section~2]{KL}, which included a classification for $k=1$ and a partial result for $k=2$. Here we complete the classification for $k=2$. For completeness, we~give the details for both cases.

\begin{Proposition} \label{prop:k1}
Let $k=1$. If $\big(M^1, \mu\big)$ is a twisted-austere pair in $\R^n$ with $M$ complete, then $L = N^* M + \mu$ is an $n$-plane in $T^* \R^n = \C^n$.
\end{Proposition}
\begin{proof}
In this case, $M^1$ is a curve. Equation~\eqref{eq:tw1} is vacuous. The~$1 \times 1$ matrix $C$ is $1 - \ri {\rm d}^* \mu$. Hence equation~\eqref{eq:tw2} becomes
\begin{gather} \label{eq:k1case}
\sin \phi = \cos \phi \, {\rm d}^* \mu.
\end{gather}
(There is a harmless sign error here in~\cite{KL}.) Using~\eqref{eq:tw3-special1} for~\eqref{eq:tw3} in the $j=1$ case (which is the only allowed value of $j$ here), and since $A^{\nu}$ is a scalar, we~get
\begin{gather*}
A^{\nu} = 0 \qquad \text{for all $\nu$}.
\end{gather*}
Thus $M^1$ is totally geodesic, hence a straight line. Without loss of generality, we~take it to be the $x$-axis in $\R^n$. Since $M = \R$, we~have $\mu = {\rm d}f$ for some $f \in C^{\infty} (\R)$. Then equation~\eqref{eq:k1case} says that $f''(x) = -\tan \phi$. Hence $\mu = (a x + b) {\rm d}x$ for some constants $a$, $b$, and $N^* M + \mu$ is an affine translation of $N^* M$ in $\C^n = \R^n \oplus \R^n$, and is thus an $n$-plane.
\end{proof}

\begin{Remark} 
Proposition~\ref{prop:k1} is consistent with Proposition~\ref{prop:totally-geodesic}, as a special Lagrangian graph in $T^* \R^1 = \C^2$ is straight line.
\end{Remark}

\begin{Proposition} \label{prop:k2}
Let $k=2$, and let $\big(M^2, \mu\big)$ be a twisted-austere pair in $\R^n$, such that $M$ is not totally geodesic. Then $\sin \phi = 0$, and $M$ is a minimal surface in $\R^n$ with $\mu$ a closed and coclosed $1$-form on $M$ with respect to the induced metric $($and hence harmonic$)$.
\end{Proposition}
\begin{proof}
In this case, $M^2$ is a surface and now $\sigma_2 = \det$. From $\det C = \det (I + \ri B) = 1 + \ri \tr B - \det B$, we~find that~\eqref{eq:tw2} becomes
\begin{gather} \label{eq:k2case-1}
\sin \phi (1 - \det B) + \cos \phi (\tr B) = 0.
\end{gather}
(There is again a harmless sign error in~\cite[equation (2.15)]{KL}.) Using~\eqref{eq:tw3-special1} for~\eqref{eq:tw3} in the $j=1$ case gives
\begin{gather} \label{eq:k2case-2}
\tr \big(A^{\nu} \big(I + B^2\big)^{-1}\big) = 0 \qquad \text{for all $\nu$}.
\end{gather}
We also have
\begin{gather*}
\frac{1}{\det C} = \frac{1}{1 + \ri \tr B - \det B} = \frac{(1 - \det B) - \ri \tr B}{(1- \det B)^2 + (\tr B)^2}.
\end{gather*}
From the above, using~\eqref{eq:tw3-specialk} for~\eqref{eq:tw3} in the $j=k=2$ case gives
\begin{gather} \label{eq:k2case-3}
(\det A^{\nu}) (\tr B) = 0.
\end{gather}
Suppose that $\tr B \ne 0$, so that $\det A^{\nu} = 0$ for all $\nu$. Fixing a particular normal direction $\nu$, we~can choose an orthonormal frame at a point on $M$ such that
\begin{gather*}
A^{\nu} = \begin{pmatrix} 0 & 0 \\ 0 & a_{22} \end{pmatrix}, \qquad
B = \begin{pmatrix} b_{11} & b_{12} \\ b_{12} & b_{22} \end{pmatrix}.
\end{gather*}
Then we have
\begin{gather*}
\big(I + B^2\big)^{-1} = \frac{1}{\det \big(I + B^2\big)}
\begin{pmatrix} 1 + b_{22}^2 + b_{12}^2 & -b_{12}(b_{11} + b_{22}) \\ -b_{12}(b_{11} + b_{22}) & 1 + b_{11}^2 + b_{12}^2 \end{pmatrix},
\end{gather*}
and thus equation~\eqref{eq:k2case-2} for this $\nu$ gives $a_{22} \big(1 + b_{11}^2 + b_{12}^2\big) = 0$, hence $a_{22} = 0$ and $A^{\nu} = 0$ for this $\nu$.
Therefore whenever $\det A^{\nu} = 0$, we~have $A^{\nu} = 0$. Since this holds for all $\nu$, we~are in the totally geodesic case which is covered by Proposition~\ref{prop:totally-geodesic}.

Therefore we can assume there exists at least one $\nu$ such that $\det A^{\nu} \neq 0$. From~\eqref{eq:k2case-3} we~deduce that $\tr B = 0$, so $\mu$ is coclosed with respect to the induced metric. Since $\mu$ is also closed by~\eqref{eq:tw1}, we conclude that $\mu$ is harmonic. Now choose at~a~point an orthonormal frame in which $B$ is diagonal. Since $\tr B = 0$, in such a frame we~have
\begin{gather*}
B = \begin{pmatrix} \lambda & 0 \\ 0 & - \lambda \end{pmatrix}.
\end{gather*}
But then $I + B^2$ is a positive scalar multiple of the identity, so~\eqref{eq:k2case-2} implies that $\tr A^{\nu} = 0$ for~all~$\nu$, so $M^2 \subset \R^n$ is a minimal surface. Finally, equation~\eqref{eq:k2case-1} becomes $(\sin \phi) \big(1 + \lambda^2\big) = 0$, so~$\sin \phi = 0$.
\end{proof}

\begin{Example} 
We illustrate Proposition~\ref{prop:k2} with an explicit example when $n=3$. Throughout this example we identify vector fields and $1$-forms on $\R^3$ using the Euclidean metric. Let $M^2$ be a surface in $\R^3$ which is given by the graph of a smooth function $h\colon \Omega \to \R$ of two variables, where $\Omega$ is some open set in $\R^2$. It~is well known that the minimal surface equation in this case~is
\begin{gather} \label{eq:minimal-ex}
\big(1 + h_v^2\big) h_{uu} + \big(1 + h_u^2\big) h_{vv} - 2 h_u h_v h_{uv} = 0.
\end{gather}
With respect to the global frame of tangent vector fields given by $\vv_1 = ( 1, 0, h_u)$ and $\vv_2 = (0, 1, h_v)$, the induced metric on $M^2$ from the Euclidean metric on $\R^3$ is
\begin{gather*}
g = \begin{pmatrix} 1 + h_u^2 & h_u h_v \\ h_u h_v & 1 + h_v^2 \end{pmatrix}
\end{gather*}
and one can compute that for a function $f\colon \Omega \to \R$, thought of as function on the Riemannian manifold $(M, g)$, and writing the coordinates on $\Omega \subseteq \R^2$ as $(u_1, u_2) = (u,v)$, its exterior derivative is
\begin{gather}
{\rm d} f = f_u \big(g^{11} \vv_1 + g^{12} \vv_2\big) + f_v \big(g^{21} \vv_1 + g^{22} \vv_2\big) \nonumber
\\ \hphantom{{\rm d} f}
\label{eq:exterior-ex}
 {}= \frac{1}{\det g} \big( \big(1 + h_v^2\big) f_u - h_u h_v f_v, - h_u h_v f_u + \big(1 + h_u^2\big) f_v, h_u f_u + h_v f_v \big),
\end{gather}
and its Laplacian is
\begin{gather}
\nonumber
\Delta_g f = \frac{1}{\sqrt{\det g}} \frac{\partial}{\partial u_i}
\bigg(g^{ij} \sqrt{\det g} \frac{\partial f}{\partial u_j} \bigg)
\\ \hphantom{\Delta_g f}
\nonumber
{}= \frac{1}{\big(1 + h_u^2 + h_v^2\big)} \big( \big(1+ h_v^2\big) f_{uu} + \big(1 + h_u^2\big) f_{vv} - 2 h_u h_v f_{uv} \big)
\\ \hphantom{\Delta_g f=}
\label{eq:Laplace-ex}
{} - \frac{1}{(\det g)^2} (h_u f_u + h_v f_v ) \big( \big(1 + h_v^2\big) h_{uu} + \big(1 + h_u^2\big) h_{vv} - 2 h_u h_v h_{uv} \big).
\end{gather}
Substituting~\eqref{eq:minimal-ex} into~\eqref{eq:Laplace-ex} eliminates the second term. We~deduce that $f$ is a harmonic function on the minimal surface $M$ if and only if
\begin{gather} \label{eq:reduced-ex}
\big(1 + h_v^2\big) f_{uu} + \big(1 + h_u^2\big) f_{vv} - 2 h_u h_v f_{uv} = 0.
\end{gather}
Using the Euclidean metric to identify covectors with tangent vectors, the conormal space is spanned by $\nu^* = (-h_u, -h_v, 1)$, and we obtain from~\eqref{eq:twisted-austere} and~\eqref{eq:exterior-ex} that the twisted conormal bundle $N^* M + d f$ is identified with the submanifold
\begin{gather*}
\{ (x_1(t,u,v), x_2(t,u,v), x_3(t,u,v), y_1(t,u,v), y_2(t,u,v), y_3(t,u,v)) : (u,v) \in \Omega, t \in \R \}
\end{gather*}
in $\R^6 \cong \C^3$, with coordinate functions given by
\begin{gather*}
x_1 = u, \qquad
x_2 = v, \qquad
x_3 = h(u,v), \\
y_1 = -t h_u + \frac{1}{1 + h_u^2 + h_v^2} \big( \big(1 + h_v^2\big) f_u - h_u h_v f_v\big), \\
y_2 = -t h_v + \frac{1}{1 + h_u^2 + h_v^2} \big({-}h_u h_v f_u + \big(1 + h_u^2\big) f_v \big), \\
y_3 = t + \frac{1}{1 + h_u^2 + h_v^2} (h_u f_u + h_v f_v).
\end{gather*}
Proposition~\ref{prop:k2} says that if the two functions $h$ and $f$ satisfy the pair of equations~\eqref{eq:minimal-ex} and~\eqref{eq:reduced-ex}, then the immersion of the open set $\Omega \times \R$ in $\R^2 \times \R$ is a special Lagrangian submanifold of $\C^3$ with phase ${\rm e}^{{\rm i} \frac{\pi}{2}}$.

Note that in particular, if we choose $f=h$ then the pair of equations~\eqref{eq:minimal-ex} and~\eqref{eq:reduced-ex} coincide. For example, we~can take $h(u,v) = \arctan{\frac{v}{u}}$, so that $M$ is a~helicoid in $\R^3$, which is a~minimal surface. Then taking $f = h$, one can compute that $(y_1, y_2, y_3)$ is
\begin{gather*}
\left(\frac{v \big(t\big(1+u^2+v^2\big)-\big(u^2+v^2\big) \big)}{\big(u^2+v^2\big)\big(1+u^2+v^2\big)},
-\frac{u \big(t(1+u^2+v^2)-\big(u^2+v^2\big) \big)}{\big(u^2+v^2\big)\big(1+u^2+v^2\big)},
\frac{1 + t\big(1+u^2+v^2\big)}{\big(1+u^2+v^2\big)} \right).
\end{gather*}
The authors verified directly that the above is a special Lagrangian submanifold of $\C^3$ with phase ${\rm e}^{{\rm i} \frac{\pi}{2}}$. Of course, even over the helicoid, there are infinitely many more solutions. Given $h(u,v) = \arctan{\frac{v}{u}}$, a computation on Maple shows that the general solution to~\eqref{eq:reduced-ex} is
\begin{gather*}
f = A_1 \left( \arctan \frac{v}{u} + \frac{1}{2} \arcsin\big(1 + 2 u^2 + 2 v^2\big) \right)\\
\hphantom{f =}{}
 + A_2 \left( \arctan \frac{v}{u} - \frac{1}{2} \arcsin\big(1 + 2 u^2 + 2 v^2\big) \right),
\end{gather*}
where $A_1$, $A_2$ are arbitrary $C^2$ functions of one variable. The~solution $f=h=\arctan \frac{v}{u}$ corresponds to $A_1 (s) = A_2 (s) = \frac{1}{2} s$.
\end{Example}

\section{Twisted-austere 3-folds} \label{sec:tw3folds}

Because the special Lagrangian $n$-folds for $M$ totally geodesic arise by taking products of lower-dimensional examples with a flat factor, we~generally exclude the case where $M$ is totally geodesic from now on.

In this section we state and prove the first of our two main theorems, which characterizes a~twisted-austere pair $\big(M^3, \mu\big)$ when $M$ is a $3$-dimensional submanifold of $\R^n$ that is not totally geodesic. There are only two possibilities.

\begin{Theorem} \label{theoremone}
Let $(M,\mu)$ be a twisted-austere pair where $M^3 \subset \R^n$ is \emph{not totally geodesic}, and let $\phi$ be as in~\eqref{angleeq} with $k=3$.
Then $\cos \phi \ne 0$, and either
\begin{enumerate}[$(i)$]\itemsep=0pt
\item $n$ is arbitrary and $M$ is {ruled by lines}, or else
\item $n=5$ and $M$ is a {generalized helicoid swept out by planes in $\R^5$}.
\end{enumerate}
\end{Theorem}

The proof of Theorem~\ref{theoremone} takes up this entire section, and we break it up into a sequence of~pro\-positions, all of which share the assumptions of Theorem~\ref{theoremone}.

\begin{Proposition} \label{thm1step1}
We have $\det (A^\nu) = 0$ for all normal directions $\nu$, and moreover $\cos \phi \ne 0$.
\end{Proposition}
\begin{proof}
Recall from Theorem~\ref{theoremzero} that, in addition to ${\rm d} \mu = 0$ which just says that $B$ is symmetric, the twisted-austere conditions for $3$-dimensional $M$ are
\begin{gather}
\Im \big({\rm e}^{\ri \phi} \det C\big) = 0, \label{condb0} \\
\Re \big(\sigma_1 \big(A^\nu C^{-1}\big) \big) = 0, \label{condb1} \\
\Im \big(\sigma_2 \big(A^\nu C^{-1}\big) \big) = 0, \label{condb2} \\
\Re \big(\sigma_3 \big(A^\nu C^{-1}\big) \big) = 0, \label{condb3}
\end{gather}
where $C = I + \ri B$. Here $\sigma_3$ is the determinant. Note that $\det C \neq 0$. (See the proof of~Pro\-position~\ref{prop:sigma-adj-1}.)

Using $\Re \det C = 1-\sigma_2(B)$ and $\Im \det C = \sigma_1(B) - \sigma_3(B)$, the first condition~\eqref{condb0} ex\-pands~as
\begin{gather} \label{expandet}
(1-\sigma_2(B)) \sin \phi + (\sigma_1(B) - \sigma_3(B))\cos\phi = 0.
\end{gather}
Next, note that we can expand conditions~\eqref{condb1} and~\eqref{condb2} using the identities
\begin{gather}
\sigma_1\big(A^\nu C^{-1}\big) = \dfrac{\sigma_1(A^\nu (I -\adj B)) + 2\ri \{A^\nu,B\}}{\det C}, \label{expandS1} \\
\sigma_2\big(A^\nu C^{-1}\big) = \dfrac{\sigma_2(A^\nu) + \ri \sigma_1( B \adj A^\nu) }{\det C}, \label{expandS2}
\end{gather}
where $\{\,,\,\}$ denotes the symmetric bilinear form corresponding to $\sigma_2$ on the space $\calS_3$ of $3 \times 3$ symmetric matrices (that is, $\{W,W\} = \sigma_2(W)$ for all $W \in \calS_3$). A general version (for $k \times k$ matrices) of the identity~\eqref{expandS2} is proved in~Proposition~\ref{prop:sigma-adj-1} and the identity~\eqref{expandS1} is proved in~Proposition~\ref{prop:sigma-adj-2}.

Suppose that $\det(A^\nu)\ne 0$. We~will derive a contradiction. The~last condition~\eqref{condb3} implies that
\begin{gather} \label{sigmab} \Re \det C = 1-\sigma_2(B) =0.\end{gather}

By~\eqref{sigmab} $\det C$ is purely imaginary, thus substituting~\eqref{expandS1} into~\eqref{condb1} implies that $\{A^\nu, B\}=0$, while substituting~\eqref{expandS2} into~\eqref{condb2} implies that $\sigma_2(A^\nu)=0$. Together with~\eqref{sigmab}, these in turn imply that $t \mapsto B + t A^\nu$ parametrizes a line on the quadric hypersurface in $\calS_3$ defined by~$\sigma_2(W) = 1$. However, by Remark~\ref{rmk:sigma2-metric}, the signature of $\sigma_2$ on $\calS_3$ is $(1,5)$. It~is well-known (and easy to check) that this implies that the hypersurface contains no lines. Hence $A^\nu=0$, which contradicts our assumption that $\det A^\nu \ne 0$.

Now suppose that $\cos\phi =0$. Then~\eqref{expandet} implies again that $\Re \det C = 1-\sigma_2(B) = 0$, so as before we conclude that $A^\nu=0$, contradicting our assumption that $M$ is not totally geodesic.
\end{proof}

For use below, we~note that multiplying the numerator and denominator of the right-hand side of~\eqref{expandS1} by ${\rm e}^{\ri \phi}$, and using the
fact that by~\eqref{condb0} the denominator is now real, we~see that~\eqref{condb1} is equivalent to
\begin{gather} \label{condc1}
\sigma_1(A^\nu (I -\adj B)) \cos \phi - 2 \{A^\nu,B\}\sin\phi = 0.
\end{gather}
Similarly, assuming~\eqref{condb0} shows that condition~\eqref{condb2} is equivalent to
\begin{gather} \label{condc2}
\sigma_2(A^\nu) \sin\phi + \sigma_1( B \adj A^\nu) \cos\phi = 0.
\end{gather}

\begin{Lemma} \label{no-rank-one}
The second fundamental form $A^{\nu}$ cannot have rank one for any normal direc\-tion~$\nu$.
\end{Lemma}
\begin{proof}
Suppose $A^{\nu}$ has rank one for some $\nu$. We~will obtain a contradiction. There is a frame with respect to which $A^{\nu} = A_0$ is of the form
\begin{gather*}
A_0 = \begin{pmatrix} 1 & 0 & 0 \\ 0 & 0 & 0 \\ 0 & 0 & 0 \end{pmatrix}.
\end{gather*}
This form is invariant under rotating the vectors $\ve_2$ and $\ve_3$ within the plane they span, so we may also assume without loss of generality that $B_{12} = 0$.

Substituting $A^\nu = A_0$ into~\eqref{condc1} gives
\begin{gather} \label{condc1case}
(B_{22} + B_{33}) \sin\phi + \big(B_{22}B_{33}-B_{23}^2-1\big) \cos\phi = 0.
\end{gather}
Equation~\eqref{condb0} in this case becomes
\begin{gather}
\big( \big(1 + B_{23}^2 - B_{22}B_{33}\big) B_{11} + B_{13}^2 B_{22} + B_{22} + B_{33} \big) \cos \phi \nonumber
\\ \qquad
{} - \big( (B_{22} + B_{33}) B_{11} - B_{13}^2 + B_{22}B_{33} - B_{23}^2 - 1 \big) \sin \phi = 0.
 \label{condb0v2}\end{gather}
Multiplying~\eqref{condc1case} by $(B_{11} + \tan \phi)$ and adding this to~\eqref{condb0v2} yields, after some manipulation, that
\begin{gather*}
B_{13}^2 (\sin \phi + B_{22}\cos\phi) + (B_{22} + B_{33})\sec\phi = 0.
\end{gather*}
Solving this equation for $B_{33}$ and substituting back into~\eqref{condc1case} gives
\begin{gather*}
\big( B_{13}^2 (\sin \phi + B_{22}\cos\phi)^2 + B_{22}^2 + B_{23}^2 + 1\big) \cos\phi = 0,
\end{gather*}
which, since $\cos\phi \ne 0$, has no real solutions.
\end{proof}

\begin{Proposition} \label{thm1step2}
At each point $p \in M$, there exists an orthonormal frame with respect to which the span
\begin{gather*}
|\II_p| = \{ \nu \cdot \II \,|\, \forall\, \nu \in N_p M\} \subseteq S^2 T^*_p M
\end{gather*}
lies in one of the following subspaces:
\begin{gather*}
\text{$(i)$\ } \W_1 = \left\{\!\!\begin{pmatrix} * & * & 0 \\ * & * & 0 \\ 0 & 0 & 0\end{pmatrix} \!\!\right\} ,\qquad
\text{$(ii)$\ } \W_2 = \left\{\!\! \begin{pmatrix} 0 & 0 & * \\ 0 & 0 & * \\ * & * & *\end{pmatrix} \!\!\right\} .
\end{gather*}
Moreover, if $\dim |\II_p| = 1$, then we are necessarily in case $(i)$.
\end{Proposition}

\begin{proof} The two possible forms for $|\II|$ follow from Proposition~\ref{sing3max} in Appendix~\ref{sec:app-singular}, and the final statement is established in the first paragraph of the proof of Proposition~\ref{sing3max}.
\end{proof}

\begin{Proposition} \label{thm1step3}
If $M$ falls into case $(i)$ of Proposition~$\ref{thm1step2}$, then $B_{33}=-\tan\phi$ with respect to the same orthonormal frame.
If $M$ does not fall into case $(i)$ then $M$ is a generalized helicoid in $\R^5$.
\end{Proposition}
\begin{proof}
Suppose that we are in case $(i)$. Then one can compute that equation~\eqref{condc2} factors as
\begin{gather*}
\big(A^\nu_{11} A^\nu_{22}-(A^\nu_{12})^2\big) (\sin\phi + B_{33}\cos\phi) = 0.
\end{gather*}
Since $M$ is not totally geodesic, Lemma~\ref{no-rank-one} tells us that there is an $A^\nu$ that has rank two. It~follows that $B_{33}=-\tan\phi$.

Now suppose that we are not in case $(i)$. By Proposition~\ref{thm1step2} we know that $\dim |\II_p| \ge 2$. Also, by Lemma~\ref{no-rank-one} we know that $\II_p$ cannot contain any rank one matrices, so it must be two-dimensional and spanned by matrices of the form
\begin{gather*}
A_1 = \begin{pmatrix} 0 & 0 & 1 \\ 0 & 0 & 0 \\ 1 & 0 & * \end{pmatrix}, \qquad
A_2 = \begin{pmatrix} 0 & 0 & 0 \\ 0 & 0 & 1 \\ 0 & 1 & * \end{pmatrix}.
\end{gather*}
{\sloppy Substituting $A^\nu=A_1$ and $A^\nu=A_2$ into~\eqref{condc2} yields respectively $B_{22} = -\tan \phi$ and \mbox{$B_{11}=-\tan\phi$}. Using these values
and substituting $A^\nu = A_1 + A_2$ into~\eqref{condc2} yields $B_{12}=0$. Fina\-lly using these values for $B_{11}$, $B_{12}$, $B_{22}$ and substituting either $A^\nu= A_1$ or $A^\nu =A_2$ into~\eqref{condc1} yields that the $(3,3)$ entry of $(A^{\nu})$ is zero, which implies that $\tr A^\nu = 0$ for all $\nu$. Thus in fact~$M$ is minimal, and with respect to an appropriate basis, we~have
\begin{gather*}
|\II_p| \subset \V'_2= \left\{\!\! \begin{pmatrix} 0 & 0 & * \\ 0 & 0 & * \\ * & * & 0\end{pmatrix} \!\!\right\}.
\end{gather*}}\noindent
It now follows that $|\II|$ is \emph{simple} in the sense of Bryant~\cite{Br}, and hence by~\cite[Theorem 3.1]{Br} that~$M$ must be a generalized helicoid. Because $|\II_p|$ has dimension at most two, the first osculating space of $M$ at each point has dimension at most five. Moreover, because the first prolongation of $\V'_2$ has dimension zero it follows from Theorem~\ref{codim-reduction-thm} in Appendix~\ref{codim-reduction-sec} that the first osculating space of $M$ is fixed, so that $M$ lies in a 5-dimensional subspace of $\R^n$.
\end{proof}

\begin{Proposition} \label{ruledprop}
If $M$ falls into case $(i)$ of Proposition~$\ref{thm1step2}$, then $M$ is ruled by lines.
\end{Proposition}
The proof of this proposition is relatively simple, but uses the method of moving frames. Before giving the proof, we~recall some details about the frame bundle which will be needed in~the proof as well as in later sections.

Let $\F$ be the oriented orthonormal frame bundle of $\R^n$, whose fiber at a point $p$ consists of~all oriented orthonormal bases of $T_p \R^n$. We~may think of a point $u$ in the fiber as a matrix $U \in {\rm SO}(n)$ whose columns comprise the corresponding frame. The~frame bundle carries a~canonical $\R^n$-valued $1$-form $\vomega$ such that
\begin{gather} \label{defcan}
\vomega_u(\vv) = U^{-1} \pi_* \vv,
\end{gather}
where $\pi\colon\F \to \R^n$ is the basepoint map and we identify $\pi_*\vv \in T_{\pi(u)} \R^n$ with a column vector in~$\R^n$ in the usual way. (In other words, the entries of $\vomega_u(\vv)$ give the coefficients of the expansion of $\pi_* \vv$ in terms of the frame corresponding to $u$.) In what follows let $\w^r$ denote the components of $\vomega$, where $1 \le r,s,t \le n$.

Suppose $M^k \subset \R^n$ is a submanifold and $f$ is a local section of $\F\vert_M$, that is a local oriented orthonormal frame field with component vector fields $\ve_1, \dots, \ve_n$. Then it follows from~\eqref{defcan} that the $\R^n$-valued function $\vx$ on $M$ giving the position in $\R^n$ satisfies
\begin{gather} \label{deex}
{\rm d}\vx = \ve_r f^* \w^r.
\end{gather}
In particular, if the frame $f$ is \emph{adapted to $M$} in the sense that $\ve_1, \dots, \ve_k$ span the tangent space to $M$ at each point, then $f^* \w^a=0$ for $k < a\le n$.

The frame bundle also carries a matrix-valued connection form $\vOmega$, taking value in $\mathfrak{so}(n)$, which satisfies the structure equation
\begin{gather*}
{\rm d}\vomega = -\vOmega \wedge \vomega, \qquad {\rm d}\vOmega = - \vOmega \wedge \vOmega.
\end{gather*}
In terms of components, these equations read
\begin{gather} \label{streq}
{\rm d}\w^r = -\w^r_s \wedge \w^s, \qquad {\rm d}\w^r_s = -\w^r_t \wedge \w^t_s.
\end{gather}
The existence (and uniqueness) of the connection form is a special case of the existence of the Levi-Civita connection on a Riemannian manifold $N$. However, when $N=\R^n$ an easy way to~obtain the connection form, in terms of its components $\omega^r_s$, is to regard the members $e_r$ of the frame as $\R^n$-valued functions on $\F$, and resolve their exterior derivatives in terms of the frame itself:
\begin{gather} \label{dees}
{\rm d}\ve_r = \ve_s \omega^s_r.
\end{gather}
Returning to the situation of an adapted frame field $f$ along a submanifold $M^k$, it follows from~\eqref{dees} that the pullbacks of the $\w^a_j$ encode the second fundamental form of $M$:
\begin{gather} \label{fundyomega}
\II(\ve_i, \ve_j) = \big(\ve_i \intprod \tw^a_j\big) \ve_a,
\end{gather}
where we use $\ve_i$, $\ve_j$ for $1\le i,j \le k$ to denote the frame vector fields tangent to $M$, and the tilde accent denotes pullback by $f$.

\begin{proof}[Proof of Proposition~\ref{ruledprop}]
Let $f=(\ve_1, \dots, \ve_n)$ be an adapted local frame along $M$ such that with respect to the basis $\ve_1$, $\ve_2$, $\ve_3$ for $T_p M$, the space $|\II|$ assumes the form $(i)$ in Proposition~\ref{thm1step2}. Then~\eqref{fundyomega} implies that $\wt^a_3=0$ for $4 \le a \le n$. Then from~\eqref{dees} we have
\begin{gather} \label{dee3}
{\rm d}\ve_3 = \ve_1 \wt^1_3 + \ve_2 \wt^2_3.
\end{gather}
We will show that the frame vector $\ve_3$ is tangent to a ruling along $M$.

By Proposition~\ref{thm1step2}, we~can assume without loss of generality that $\ve_4 \cdot \II$ has rank two. Then
\begin{gather*}
0 = {\rm d}\wt^4_3 = -\wt^4_1 \wedge \wt^1_3 - \wt^4_2 \wedge \wt^2_3 = \wt^1_3 \wedge \big(A_{1j} \wt^j\big) + \wt^2_3 \wedge \big(A_{2j} \wt^j\big),
\end{gather*}
where $A_{ij}$ for $1\le i,j \le 2$ are the entries of a rank two matrix. Since the $1$-forms in parentheses on the right are linearly independent, we~have
\begin{gather} \label{noturn}
\wt^j_3 \equiv 0 \mod \wt^1, \wt^2.
\end{gather}
That is, $\wt^1_3$ and $\wt^2_3$ are linear combinations of $\wt^1$ and $\wt^2$. Then from~\eqref{dee3} we have ${\rm d}\ve_3 \equiv 0 \mod \wt^1, \wt^2$. Thus, $\ve_3$ is fixed as one moves along $M$ in the direction of~$\ve_3$.
\end{proof}

Before turning to the construction of examples of twisted-austere pairs, we~now derive some equations that relate the adapted moving frame (and the associated
$1$-forms) to the matri\-ces~$A^\nu$,~$B$ that satisfy the twisted-austere conditions. (These equations will be needed in the next two sections.)
First, equation~\eqref{fundyomega} can be rewritten as
\begin{gather} \label{waiA}
\tw^a_i = (A^a)_{ij} \tw^j,
\end{gather}
where the matrix $A^a$ gives the components of the second fundamental form in the direction of~$\ve_a$. Next, because the $\wt^i$ form a coframe
along $M$, we~can expand
\begin{gather*} 
\mu = \mu_i \tw^i.
\end{gather*}
Then $\nabla \mu = B_{ij} \wt^i \otimes \wt^j$, where the $B_{ij}$ are calculated using
\begin{gather} \label{demuse}
{\rm d}\mu_i - \mu_j \wt^j_i = B_{ij} \wt^j.
\end{gather}

In terms of this equation, the results of Propositions~\ref{thm1step3} and~\ref{ruledprop} can be interpreted as follows. For a base $M$ carrying an adapted moving frame with respect to which $|\II|$ lies in $\W_1$, the frame vector $\ve_3$ points along the ruling. Thus, $\ve_3 \intprod \mu = \mu_3$ is a natural geometric invariant which we will refer to as the \emph{slope} of the twisted-austere pair $(M,\mu)$. (Note that this depends on~a~choice of~orientation for the rulings.) Then using~\eqref{demuse}, along with~\eqref{noturn}, we~can interpret the condition $B_{33}=-\tan\phi$ as saying that the derivative of the slope along the ruling is equal to~the constant~$-\tan\phi$.

\section{Cylindrical examples} \label{cylindersec}

We saw in Theorem~\ref{theoremone} that if $M^3$ is the base of a twisted-austere pair, then either $M$ is ruled by~lines or is a generalized helicoid in $\R^5$ which is ruled by planes. In this section we will construct special examples of twisted-austere pairs $\big(M^3, \mu\big)$ assuming that $M$ is ruled by \emph{parallel} lines, that is $M$ is a cylinder. From now on, it will be convenient for us to take the ambient space as $\R^{n+1}$, equipped with Euclidean coordinates $x^0, x^1, \dots, x^n$ such that the rulings point in~the~$x^0$ coordinate direction. Corresponding to this, we~now change to using $\ve_0$, $\ve_1$, $\ve_2$ to~denote the members of the moving frame that are tangent to $M$, with $\ve_0$ pointing along the rulings.

Let $\Sigma_0$ be the surface obtained by intersecting $M$ with a copy of $\R^{n}$ perpendicular to the rulings. (For the sake of argument, let this $\R^{n}$ be the hyperplane given by $x^0=0$.) We can construct an adapted moving frame along $M$ by taking an adapted moving frame $\ve_1, \ve_2, \ve_3, \dots, \ve_n$ along $\Sigma_0$ (such that $\ve_1$, $\ve_2$ are tangent to the surface), parallelly translating these vectors along the rulings, and completing the frame with the constant unit vector field $\ve_0$ tangent to the rulings. In what follows, it will be convenient to take the index ranges $0 \le \alpha,\beta \le 2$, $1 \le i,j,l,m \le 2$ and $3 \le a,b \le n$; so, for example, equation~\eqref{demuse} now reads
\begin{gather} \label{demusea}
{\rm d}\mu_\alpha - \mu_\beta\wt^\beta_\alpha = B_{\alpha\beta} \wt^\beta.
\end{gather}

The canonical forms and connection forms on $M$ defined by~\eqref{deex} and~\eqref{dees} satisfy{\samepage
\begin{enumerate}[$(i)$]\itemsep=0pt
\item $\wt^1$, $\wt^2$ and $\wt^1_2$ are basic for the projection to $\Sigma_0$, and the same is true for the $\wt^a_i$,
\item because $\ve_0$ is constant on $M$, the forms $\wt^i_0$ and $\wt^a_0$ are zero,
\item as a result, the first structure equation in~\eqref{streq} implies that $\wt^0$ is closed.
\end{enumerate}
In fact, if we let $u$ be the restriction to $M$ of the ambient coordinate $x^0$, then $\wt^0={\rm d}u$.}

Suppose that, on $\Sigma_0$, we~have $\wt^a_i = h^a_{ij} \wt^j$, so that the $h^a_{ij}$ are the components of the second fundamental form of $\Sigma_0$ as a submanifold in $\R^n$. Pulling the $\wt^a_i$ back to $M$, we~see that the components of $M$'s second fundamental form are given by
\begin{gather} \label{cylinderAform}
A^a = \begin{pmatrix} 0 & 0 \\ 0 & h^a_{ij} \end{pmatrix},
\end{gather}
where now the zeros are in the first row and column, corresponding to the tangent vector~$e_0$. Since these matrices are singular, the highest-order twisted-austere condition~\eqref{condb3} holds automatically. Since $M$ is not totally geodesic, Lemma~\ref{no-rank-one} tells us that $A^a$ has rank two for at least one normal direction $e^a$, and by Proposition~\ref{thm1step3} the next-highest-order twisted-austere condition forces $B_{00}=-\tan\phi$.

We now consider the $u$-dependence of the components of $\mu$ and its covariant derivative. Because $\wt^i_0=0$ and $B_{00}=-\tan\phi$, equation~\eqref{demusea} implies that
\begin{gather*}
{\rm d}(\mu_0 + u \tan\phi) = B_{0i} \wt^i.
\end{gather*} Since the right-hand side of the above equation is semibasic for the projection to $\Sigma_0$, and recalling that $\wt^0 = {\rm d}u$, we~can write
\begin{gather} \label{mu0eq}
\mu_0 = k\sec\phi - u \tan\phi,
\end{gather}
where $k$ is a smooth function on $\Sigma_0$. Define the smooth functions $k_i$ on $\Sigma_0$ by ${\rm d}k = k_i \wt^i$, so that $B_{0i} = k_i \sec \phi$. Then~\eqref{demusea} implies that ${\rm d}(\mu_i - u k_i \sec\phi)$ is semibasic for $\Sigma_0$, so we may set
\begin{gather*}
\mu_i = \lambda_i + u k_i \sec\phi,
\end{gather*}
where the $\lambda_i$ are functions on $\Sigma_0$. (Note, however, that these depend on the choice of frame on~$\Sigma_0$, while $k$ does not.) Substituting these into
\eqref{demusea} then gives
\begin{gather*}
{\rm d}\lambda_i -\lambda_j \wt^j_i = B_{ij} \wt^j - u\sec\phi \big({\rm d}k_i -k_j \wt^j_i\big).
\end{gather*}
Expanding both sides as polynomials in $u$ and comparing coefficients, we~obtain
\begin{gather*}
B_{ij} = \lambda_{ij}+ u k_{ij} \sec\phi,
\end{gather*}
where we have set
\begin{gather*}
{\rm d}\lambda_i = \lambda_j \wt^j_i + \lambda_{ij} \wt^j, \qquad
{\rm d}k_i = k_j\wt^j_i + k_{ij} \wt^j.
\end{gather*}
The $k_{ij}$ are the components of the Hessian $\nabla^2 k$ with respect to the coframe on $\Sigma_0$, and the $\lambda_{ij}$ are also symmetric in $i$ and $j$, indicating that $\lambda = \lambda_i \wt^i$ (which is well-defined, independent of choice of coframe) is a closed $1$-form on $\Sigma_0$. In terms of these tensor components, we~have
\begin{gather} \label{cylinderBform}
B = \begin{pmatrix} -\tan \phi & k_i \sec\phi \\ k_i \sec\phi & \lambda_{ij} + u k_{ij} \sec\phi \end{pmatrix}.
\end{gather}

Substituting~\eqref{cylinderAform} and~\eqref{cylinderBform} into the two remaining twisted-austere conditions~\eqref{expandet} and~\eqref{condc1}, and equating powers of $u$, gives
\begin{gather}
\big(1 +k_2^2\big) \lambda_{11} - 2k_1k_2 \lambda_{12} + \big(1 + k_1^2\big) \lambda_{22} = -\big(k_1^2 +k_2^2\big) \tan \phi, \label{krell}
\\
\big(1+k_2^2\big) k_{11} - 2k_1 k_2 k_{12} + \big(1+k_1^2\big) k_{22} =0,
\label{krone}
\\
\big(1+k_2^2\big) h^a_{11} - 2 k_1 k_2 h^a_{12} + \big(1 + k_1^2\big) h^a_{22} = 0.
\label{krtwo}
\end{gather}
We now give a geometric interpretation of the last two equations.

\begin{Proposition}
Let $\Sigma_0 \subset \R^n$ be a surface and let $k$ be a smooth function on $\Sigma_0$. We~endow $\Sigma_0$ with the metric $g_0$ it inherits from $\R^{n}$. Let $dk$ and $\nabla^2 k$ have components $k_i$ and $k_{ij}$ respectively, and let $h^a_{ij}$ be the components of the second fundamental form of $\Sigma_0$, relative to an adapted orthonormal frame $e_1$, $e_2$, $e_a$. Let $\Sigma = \big\{ (p, k(p)) \in \R^n \times \R \,|\, p \in \Sigma_0 \big\}$ be the graph of $k$. Then $\Sigma$ is a minimal surface in $\R^{n+1}$ if and only if $k$ satisfies~\eqref{krone} and~\eqref{krtwo}.
\end{Proposition}
\begin{proof} Let $\ghat$ be the pullback to $\Sigma_0$ of the ambient metric on $\Sigma$. Then
\begin{gather} \label{jihat}
\ghat_{ij} = \delta_{ij} + k_i k_j.\end{gather}
Let $\w^i$ be the dual $1$-forms to the $e_i$, let $\w^i_j$ be the connection forms for the metric $g_0$ on $\Sigma_0$, and let $\varphi^i_j$ denote the connection forms for the metric $\ghat$ with respect to the \emph{same} coframe. Differentiating~\eqref{jihat} yields that
\begin{gather} \label{changeconn}
\varphi^i_j - \w^i_j = \big(\ghat^{-1}\big)^{i\ell } k_\ell k_{jm} \w^m.
\end{gather}
Letting $\hatnabla$ denote the covariant derivative with respect to $\ghat$, we~can compute that relative to the coframe $\w^1$, $\w^2$, we~have
\begin{gather*}
\big(\hatnabla^2 k\big)_{ij} =\dfrac{1}{\det\ghat}k_{ij}.
\end{gather*}
It follows that equation~\eqref{krone} is equivalent to $\Delta_{\ghat} k = 0$. On the other hand, equation~\eqref{krtwo} says that the trace with respect to $\ghat$ of the second fundamental form of~$\Sigma_0$ vanishes. That is, the projection $\pi\colon \Sigma \to \Sigma_0$ is harmonic.

In summary, the equations~\eqref{krone},~\eqref{krtwo} hold if and only if the coordinate functions on $\Sigma$ are harmonic (relative to $\ghat$),
which in turn is equivalent to $\Sigma$ being minimal.
\end{proof}

We now geometrically interpret the remaining equation~\eqref{krell}. Recall that the $1$-form $\lambda = \lambda_i \wt^i$ is closed. If we introduce a local potential function $\ell$ on $\Sigma_0$ such that ${\rm d}\ell = \lambda$, then using~\eqref{changeconn} one computes that
\begin{gather*}
\big(\hatnabla^2 \ell\big)_{ij} = \lambda_{ij} -\big(\ghat^{\,-1}\big)^{lm} \lambda_l k_m k_{ij}.
\end{gather*}
In particular, assuming that $k$ satisfies~\eqref{krone}, then equation~\eqref{krell} is equivalent to
\begin{gather*} 
\Delta_{\ghat} \ell = - \big|\hatnabla k\big|^2_{\ghat} \tan \phi.
\end{gather*}
Below, we~will also express this condition in terms of the codifferential of $\lambda$.

Gathering together all our conclusions in this section, we~have established the following result. Here we drop the hats and just use the metric on the graph of~$k$, referring to the graph of~$k$ as~$\Sigma$ and its induced metric from~$\R^{n+1}$ as~$g$.
\begin{Theorem} \label{cylinderthm}
Assume that $\big(M^3,\mu\big)$ is a twisted-austere pair, and that $M \subset \R^{n+1}$ is ruled by parallel lines. Then $M$ is the union of lines passing through a minimal surface $\Sigma \subset \R^{n+1}$. Moreover, if we choose Euclidean coordinates $x^0, x^1, \dots, x^n$ such that the rulings point in the $x^0$-direction, then
\begin{gather}
\mu = \pi^*\lambda \!+\! \sec\phi \, {\rm d} ( u(\pi^* k)) \!-\! \tan\phi \, u {\rm d}u
 = \pi^*\lambda \!+\! \sec\phi \, ( (\pi^* k) {\rm d}u \!+\! u {\rm d} (\pi^* k)) \!-\! \tan\phi \, u {\rm d}u,
\label{alphaform}
\end{gather}
where $u$ is the restriction of the $x^0$ coordinate to $M$, $k$ is the restriction of $x^0$ to $\Sigma$,
$\pi\colon M \to \Sigma$ is the projection along the rulings, and $\lambda$ is a closed $1$-form on $\Sigma$ satisfying
\begin{gather} \label{alphaq}
*{\rm d}\! *\!\lambda = |\nabla k|^2 \tan\phi,
\end{gather}
where the Hodge star and norms used are with respect to the metric on $\Sigma$. Conversely, given a~minimal surface $\Sigma \subset \R^{n+1}$ which is everywhere transverse to a fixed coordinate direc\-tion~$\di/\di x^0$, and a $1$-form $\lambda$ satisfying~\eqref{alphaq} for $k$ being the restriction of $x^0$ to $\Sigma$, then the union of lines through $\Sigma$ parallel to this direction gives a~$3$-dimensional submanifold $M$ which forms a~twisted-austere pair with $\mu$ given by~\eqref{alphaform}.
\end{Theorem}

\section{Examples with austere bases} \label{austeresec}

{\sloppy
In this section we will determine all examples of twisted-austere pairs $\big(M^3, \mu\big)$ where the ``base''~$M$ is \emph{austere} but is not totally geodesic, nor a generalized helicoid. By Proposition~\ref{ruledprop} and our assumption that $M$ is not a generalized helicoid, we~know that $M$ is ruled by lines. As in the previous section we let $\R^{n+1}$ be the ambient Euclidean space, and we number the orthonormal frame vectors as $(\ve_0, \ve_1, \ve_2, \ve_3, \dots \ve_n)$, where $\ve_0$, $\ve_1$, $\ve_2$ are tangent to $M$ with $\ve_0$ pointing along the rulings.

}

The fact that $M$ is ruled also follows from Bryant's classification of austere 3-folds in Euclidean space~\cite{Br}, which asserts that $M$ is either a product of a minimal surface in $\R^n$ with a line, or~a~(possibly twisted) cone over a minimal surface in the $S^n$. (The twisted cone construction will be reviewed below.)

For 3-dimensional submanifolds, the austere condition amounts to minimality and \mbox{$\det A^\nu =0$} for all normal directions~$\nu$. By Proposition~\ref{thm1step1} the twisted-austere conditions imply the determinant condition. We~will now see how the twisted-austere conditions simplify in the presence of~the minimality condition. With respect to the moving frame (adapted as described at the start of this section), the matrices representing $\nabla \mu$ and the second fundamental form in the direction of $\ve_r$ (for $3 \le r \le n$) look like
\begin{gather} \label{austereABform}
B=\begin{pmatrix} -\tan \phi & B_{10} & B_{20} \\ B_{10} & B_{11} & B_{12} \\ B_{20} & B_{12} & B_{22} \end{pmatrix} , \qquad
A^r = \begin{pmatrix} 0 & 0 & 0 \\ 0 & A^r_{11} & A^r_{12} \\ 0 & A^r_{12} & -A^r_{11} \end{pmatrix} ,
\end{gather}
respectively. (Note that we have incorporated the minimality condition.) Recall from the proof of Proposition~\ref{thm1step3} that the twisted-austere condition~\eqref{condc2} is satisfied by the constant value of the top-left entry of $B$. In terms of these matrix entries, it is straightforward to compute that the two remaining conditions~\eqref{expandet} and~\eqref{condc1} are equivalent respectively to the pair of equations
\begin{gather} \label{C0expanded}B_{11}+B_{22} + \sin\phi \cos\phi \big(B_{10}^2 + B_{20}^2\big) + \cos^2\phi \big( B_{20}^2 B_{11} -2B_{10}B_{20} B_{12} + B_{10}^2 B_{22}\big)=0
\end{gather}
and
\begin{gather} \label{C1minimal}
\big(B_{20}^2 -B_{10}^2\big) A^r_{11} -2B_{10}B_{20} A^r_{12} =0.
\end{gather}

If $B_{10}=B_{20}=0$, then conditions~\eqref{C0expanded},~\eqref{C1minimal} greatly simplify: the first becomes $B_{11}+B_{22}=0$ and the second condition becomes vacuous. In this case, the tensor $B=\nabla \mu$ splits as \mbox{$B = - \tan \phi \big(\omega^0\big)^2 + B_{ij} \omega^i \omega^j$}, which is the sum of a constant multiple of the square of the arclength element along the ruling plus a quadratic form which restricts to be zero along the rulings. We~will refer to this as the \emph{split case}, and the case where one of $B_{10}$, $B_{20}$ is always nonzero as the \emph{non-split case}.

\subsection{The split case} \label{splitsec}

We begin by defining an exterior differential system whose integral submanifolds correspond to the adapted frame described above. In what follows, we~will use index ranges $0\le a,b,c,e \le 2$, $1\le i,j,k\le 2$ and $3 \le r,s \le n$.

To an adapted frame $f$ along $M$ we can associate a submanifold of the orthonormal frame bundle $\F$ of $\R^{n+1}$ by simply taking the image of $f\colon M \to \F|M$. However, if we want to~cha\-rac\-te\-rize submanifolds satisfying the austere conditions, we~must introduce the components $A^r_{ab}$ of the second fundamental form as extra variables, and take the image of $(f,A)$ which is a submanifold of $\F \times \calT_1$, where $\calT_1 = \calS_3 \otimes \R^{n-2}$ is the space of $\R^{n-2}$-valued symmetric bilinear forms on $\R^3$. For example, if we were investigating submanifolds $M^3$ whose second fundamental form satisfies certain algebraic conditions that defined a smooth subvariety $N \subset \calT_1$, then on~$\F\times N$ we would define $1$-forms
\begin{gather*}
\w^r_a - A^r_{ab} \w^b
\end{gather*}
(where the components $A^r_{ab}$ are taken as coordinate functions on $\calT_1$) which, due to equation~\eqref{waiA}, would pull back to be zero on the image of $(f,A)$ when $M$ satisfies the conditions.

In our situation we want to impose conditions which also involve $\mu$, so we need to introduce the components of $\mu$ and $\nabla \mu$ as additional variables. Accordingly, let $\calT_2 = \calT_1 \times \R^3 \times \calS_3$ denote the space where the tensor components $(A^r_{ab}, \mu_a, B_{ab})$ take values, and let $N \subset \calT_2$ be the affine subspace defined by
\begin{gather} \label{splitedseqns}
B_{11}+B_{22}=0, \qquad
B_{00}=-\tan\phi, \qquad
B_{0i}=0,\qquad
A^r_{0a}=0, \qquad
A^r_{11}+A^r_{22}=0.
\end{gather}
(Thus, $N$ has dimension $5+2(n-2)$.) On $\F \times N$ define $1$-forms $\vbeta$, $\vtheta$, $\vOmega$
(taking value in $\R^3$, $\R^{n-2}$, and the space of $(n-2)\times 3$ matrices, respectively) as follows:
\begin{subequations} \label{bigsysdef}
\begin{gather}
\beta_a := -{\rm d}\mu_a + \mu_b \w^b_a + B_{ab} \w^b,
\label{betadef} \\
\vtheta := \big(\w^3, \dots, \w^n\big)^{\rm T}, \label{vthetadef}
\\
\Omega^r_a := -\w^r_a + A^r_{ab} \w^b. \label{Omegaradef}
\end{gather}
\end{subequations}
Then if $\big(M^3,\mu\big)$ is a twisted-austere pair where the base is austere and of split type, the image of~$(f,A,\mu,B)$ is an integral submanifold of the Pfaffian system $\calI$ generated by $\vbeta$, $\vtheta$, $\vOmega$. Because this integral submanifold lies over $M \subset \R^{n+1}$, it satisfies the independence condition $\w^0 \wedge \w^1 \wedge \w^2\ne 0$, and we will refer to integral submanifolds satisfying this condition as {\em admissible}. Conversely, any admissible integral submanifold of $\calI$ is generated by a moving frame along an austere $M^3 \subset \R^{n+1}$ such that $(M,\mu)$ is a twisted-austere pair of split type.

\begin{Lemma} \label{pqprolong}
On any admissible integral submanifold $\Mhat^3$ of $\calI$, there are functions $p$, $q$ such that
\begin{gather} \label{austpq}
\begin{aligned}
\w^1_0 &= p \w^1 + q \w^2,
\\
\w^2_0 &= -q \w^1 + p \w^2.
\end{aligned}
\end{gather}
Moreover, the corresponding submanifold $M \subset \R^{n+1}$ is one of the following three possibilities:
\begin{enumerate}[$(i)$]\itemsep=0pt
\item a product of a line with a surface in $\R^n$ when $p=q=0$,
\item a cone over a minimal surface in $S^n$ when $q=0$ but $p\ne0$, or
\item a twisted cone when $q\ne 0$.
\end{enumerate}
Moreover, in case (iii), integral submanifolds only exist if $\sin\phi =0$.
\end{Lemma}
\begin{proof}
The system $\calI$ is algebraically generated by the component $1$-forms of $\vbeta$, $\vtheta$, $\vOmega$ and their exterior derivatives, and for $\Mhat$ to be an integral submanifold it is necessary and sufficient that this finite list of $1$-forms and 2-forms pull back to be zero on $\Mhat$. Moreover, in computing the generator 2-forms, any terms which are wedge products with the generator $1$-forms may be omitted. (This is known as computing `modulo the system $1$-forms', denoted by $\calI_1$.) For~example, ${\rm d}\w^r = -\Omega^r_a \wedge \w^a \equiv 0$ modulo $\calI_1$, and hence the exterior derivative of the components of~$\vtheta$ do not contribute any additional generator 2-forms to $\I$. In the same way, using~\eqref{streq},~\eqref{vthetadef}, and~\eqref{Omegaradef} we compute
\begin{align*}
{\rm d} \Omega^r_a & = - {\rm d} \omega^r_a + {\rm d} A^r_{ab} \wedge \omega^b + A^r_{ab} {\rm d} \omega^b
\\
& = \omega^r_c \wedge \omega^c_a + \omega^r_s \wedge \omega^s_a + {\rm d} A^r_{ab} \wedge \omega^b - A^r_{ab} \big(\omega^b_c \wedge \omega^c + \omega^b_s \wedge \omega^s\big)
\\
& \equiv A^r_{cb} \omega^b \wedge \omega^c_a + \omega^r_s \wedge A^s_{ac} \omega^c + {\rm d} A^r_{ab} \wedge \omega^b - A^r_{ab} \omega^b_c \wedge \omega^c \mod \calI_1
\\
& \equiv \big( {\rm d} A^r_{ab} -A^r_{ac} \w^c_b -A^r_{cb} \w^c_a + A^s_{ab} \w^r_s\big) \wedge \w^b \mod \calI_1.
\end{align*}
In particular, noting the zero entries in $A^r$ from~\eqref{splitedseqns} gives ${\rm d} \Omega^r_0 \equiv -A^r_{ij} \w^i_0 \wedge \w^j$. Because $A^r$ has rank two for at least one $r$, it follows from the Cartan Lemma that on any admissible integral submanifold we have $\w^i_0= P^i_j \w^j$ for some functions $P^i_j$. Substituting this into the expression for~${\rm d} \Omega^r_0$, we~then find that in order for the 2-form ${\rm d} \Omega^r_0$ to vanish along $\Mhat$, we~must have
\begin{gather} \label{APcond}
A^r_{ik} P^k_j = A^r_{jk} P^k_i
\end{gather}
for all $i$, $j$. In other words, if we think of the $P^i_j$ as entries in a $2\times 2$ matrix, then $AP$ must be symmetric whenever $A$ is the lower-right block of a matrix~$A^r$.

Suppose first that $|\II|$ is $2$-dimensional on an open subset of $M$. Then the span of the lower-right blocks of the $A^r$ includes $\left(\begin{smallmatrix} 1 & 0 \\ 0 & -1 \end{smallmatrix}\right)$ and $\left(\begin{smallmatrix} 0 & 1 \\ 1 & 0 \end{smallmatrix}\right)$, and substituting these into~\eqref{APcond} shows that $P$ must have the form $P = \left(\begin{smallmatrix} p & q \\ -q & p \end{smallmatrix}\right)$ for some functions $p$, $q$, as claimed.

On the other hand, suppose that $|\II|$ is $1$-dimensional on $M$. Then we may adapt the frame so that, say $A^3 = \lambda \left(\begin{smallmatrix} 0 & 0 & 0 \\
0 & 0 & 1 \\ 0 & 1 & 0 \end{smallmatrix}\right)$ for some $\lambda\ne0$, and $A^t=0$ for $t>3$. In this case, one can compute that
\begin{gather*}
{\rm d} \Omega^3_1 \wedge \w^2 - {\rm d} \Omega^3_2 \wedge \w^1 \equiv \lambda \big(\w^2_0 \wedge \w^2 - \w^1_0 \wedge \w^1\big) \wedge \w^0
\end{gather*}
modulo the 1-forms in $\calI$. The left hand side of this expression is in $\calI$, and thus the right hand side must vanish when pulled back to $\widehat{M}$. Substituting $\w^i_0= P^i_j \w^j$ into the right hand side gives $P^1_2 = - P^2_1$. Using~\eqref{APcond} shows that $P_1^1 = P_2^2$. Thus $P$ indeed has the desired form.

From~\eqref{dees}, and using~\eqref{Omegaradef} and~\eqref{splitedseqns}, the differential of the unit vector parallel to the ruling on $M$ is
\begin{align*}
{\rm d}\ve_0 &\equiv \ve_1\w^1_0 + \ve_2 \w^2_0 \mod \calI, \\
&= p \big(\ve_1 \w^1 + \ve_2 \w^2\big) + q \big(\ve_1 \w^2-\ve_2 \w^1\big) \quad \text{using~\eqref{austpq}.}
\end{align*}
Comparing this with the differential of the position vector on $M$, given by~\eqref{deex}, we~see that
\begin{gather*}
{\rm d}\ve_0 - p\, {\rm d}\vx \equiv q \big(\ve_1 \w^2-\ve_2 \w^1\big) \mod \calI, \w^0.
\end{gather*}
In other words, if we follow a curve on $M$ orthogonal to the rulings, the differential of the ruling direction $\ve_0$ is proportional to the differential of the position
if and only if $q$ is identically zero, and the ruling direction is constant if and only if $p$ and $q$ are both identically zero. Hence, $M$~is a cylinder iff $p=q=0$ and is a cone iff $q=0$ and $p\ne 0$. That the remaining case, where $q$ is nonvanishing, corresponds to $M$ being a twisted cone follows from Bryant's classification~\cite{Br}.

Additional generator 2-forms for $\calI$ are obtained by differentiating~\eqref{betadef} and using the three equations in~\eqref{bigsysdef}. One computes that
\begin{gather} \label{dbetacomp}
{\rm d}\beta_a \equiv \big({\rm d}B_{ab} -B_{ac} \w^c_b - B_{cb} \w^c_a\big) \wedge \w^b + \mu_b A^r_{bc} A^r_{ae} \w^c \wedge \w^e \quad \mod \calI_1.
\end{gather}
In particular, from~\eqref{austereABform} we obtain that
\begin{gather*}
{\rm d}\beta_0 \equiv \tan\phi \,\w^0_i \wedge \w^i -B_{ij} \w^i_0 \wedge \w^j.
\end{gather*}
Substituting~\eqref{austpq} into the above shows that this 2-form equals $-2q \tan \phi \, \w^1 \wedge \w^2$. Thus, admissible integral manifolds with $q\ne 0$ exist only if $\tan\phi=0$.
\end{proof}

We now consider the three sub-cases given by Lemma~\ref{pqprolong}. In what follows, we~will let $m = \ve_0 \intprod \mu = \mu_0$ denote the slope, so that
\begin{gather*}
\mu = m \w^0 + \mu_i \w^i.
\end{gather*}
Note that from the form of $\beta_0$ in~\eqref{betadef} we have
\begin{gather*}
{\rm d}m \equiv \mu_i \w^i_0 - \tan \phi \w^0 + B_{0i} \w^i \quad \mod \I_1.
\end{gather*}

\subsubsection[M is a cylinder]
{$\boldsymbol M$ is a cylinder} \label{splitcylindersec}

In this case, we~may write $M = \Sigma \times \R$ where $\Sigma$ is a minimal surface in $\R^n$. We~let $t$ be the coordinate on the $\R$-factor, and write ${\rm d}t$ to denote the pullback to $M$ of its differential, which coincides with the dual $\w^0$ of the frame vector $\ve_0$. Since the $\w^i_0=0$ on $\Mhat$, equation~\eqref{betadef} gives $\beta_0 = -({\rm d}m + \tan \phi\, {\rm d}t)$, and the vanishing of this $1$-form implies that $m+t \tan \phi $ is constant. Furthermore, wedging~\eqref{betadef} with $\w^a$ and summing over $a$ gives ${\rm d}\big(\mu_1 \w^1 + \mu_2 \w^2\big) \equiv 0$ modulo~$\beta_1$,~$\beta_2$, which shows that $\mudown = \mu - m\, {\rm d}t$ is a well-defined closed $1$-form on $\Sigma$, and it is easy to check that~$\mudown$ is harmonic. The~converse also holds:

\begin{Theorem} Let $\Sigma$ is an arbitrary minimal surface in $\R^n$, $\mudown$ a harmonic $1$-form on $\Sigma$ and let~$m$ be a linear function with derivative $-\tan\phi$. Then the cylinder $M = \Sigma \times \R \subset \R^{n+1}$, together with $\mu = \mudown + m(t) \,dt$ where $t$ is the coordinate on the second factor, forms a twisted-austere pair.
\end{Theorem}

\begin{proof} This is a special case of the construction in Theorem~\ref{cylinderthm} with $k$ constant.
\end{proof}

\subsubsection[M is a cone]
{$\boldsymbol M$ is a cone} \label{splitconesec}

In this case we assume that $q=0$ identically. We~will show in Remark~\ref{prmk} below that $p$ is nowhere zero, and thus up to a change of orientation we can assume that $p > 0$ everywhere. To analyze this case, we~first construct a partial prolongation of the system $\calI$, introducing the components of $\w^0_1$, $\w^0_2$ and ${\rm d}p$ as new variables.
To this end, let $p$ be a coordinate on the last factor in $\F\times N\times \R^+$, and on this space define $1$-forms
\begin{gather} \label{coneadd78}
\alpha_1 := \w^0_1 + p \w^1, \qquad \alpha_2 := \w^0_2 + p\w^2.
\end{gather}
One can compute that
\begin{gather*}
{\rm d}\alpha_1 \equiv \big({\rm d}p+p^2 \w^0\big)\wedge \w^1, \qquad
{\rm d}\alpha_2 \equiv \big({\rm d}p+p^2 \w^0\big)\wedge \w^2
\end{gather*}
modulo $\calI$, $\alpha_1$, $\alpha_2$, and thus
\begin{gather} \label{dpeq}
{\rm d}p = -p^2 \w^0
\end{gather}
on any admissible integral manifold.

\begin{Remark} \label{prmk}
It follows from~\eqref{dpeq} and the equations~\eqref{deex},~\eqref{dees} that the vector $\vx - (1/p) \ve_0$ is constant, giving the position of the vertex of the cone. Equation~\eqref{dpeq} also implies that $p$ is constant along surfaces orthogonal to the rulings, whereas along the rulings it behaves like solutions to the separable ODE ${\rm d}y/{\rm d}s = -y^2$, for which $1/y$ is a linear function of $s$. Therefore~$p$ cannot vanish, but it can blow up to infinity, which happens at the vertex of the cone.
\end{Remark}

Let $\alpha_3 := {\rm d}p+p^2 \w^0$ and let $\calI^+$ be the Pfaffian system generated by $\valpha=(\alpha_1,\alpha_2,\alpha_3)$, $\vbeta$, $\vtheta$ and $\vOmega$.

\begin{Lemma} \label{conelemma}
Admissible integral manifolds of $\calI^+$ exist only if $\tan\phi=0$.
\end{Lemma}
\begin{proof}
Using~\eqref{dbetacomp} and the identity $B_{11} + B_{22} = 0$ which holds in the split case, a lengthy computation gives
\begin{gather*}
{\rm d}\beta_1 \wedge \w^2 - {\rm d}\beta_2 \wedge \w^1 \equiv 2p \tan\phi \, \w^0 \wedge \w^1 \wedge \w^2 \mod \calI^+_1. \tag*{\qed}
\end{gather*}
\renewcommand{\qed}{}
\end{proof}

Using $s=1/p$ to denote the function on $M$ giving the distance to the vertex of the cone, and recalling from Lemma~\ref{conelemma} that because we must have $\tan \phi = 0$,  one can compute using~\eqref{betadef} that
\begin{gather*}
\mu \equiv {\rm d}(ms) \quad \mod \calI^+.
\end{gather*}
Using~\eqref{betadef} and~\eqref{austpq}, we~obtain ${\rm d}m = \mu_0 \w^i_0 = p\big(\mu_1 \omega^1 + \mu_2 \omega^2\big)$ on solutions. In particular, ${\rm d}m$ has no $\w^0$ component, so the slope $m$ is constant along the rulings, and is thus a well-defined function on $\Sigma$.

\begin{Theorem}
The slope satisfies $\Delta m =-2m$, where $\Delta$ denotes the Laplacian on $\Sigma$. Conversely, if $\Sigma$ is an arbitrary minimal surface in $S^n$ and $m$ is a smooth function on $\Sigma$ satisfying $\Delta m = -2m$, then the cone over $\Sigma$ together with $\mu = {\rm d}(ms)$ is a twisted-austere pair for $\tan\phi=0$.
\end{Theorem}
\begin{proof}
Modulo the $1$-forms of $\calI^+$ we can compute that
\begin{gather*}
{\rm d}m \equiv p\big(\mu_1 \w^1 + \mu_2 \w^2\big),
\\
{\rm d}\mu_1 -\mu_2 \w^2_1 \equiv (B_{11} -mp) \w^1 + B_{12} \w^2,
\\
{\rm d}\mu_2 +\mu_1 \w^2_1 \equiv B_{12} \w^1 + (B_{22} - mp) \w^2,
\end{gather*}
Setting $p=1$ to restrict to $\Sigma$, we~compute using the above and~\eqref{streq} that
\begin{gather*}
* \Delta m = {\rm d} * {\rm d} m = {\rm d} \big( \mu_1 \w^2 - \mu_2 \w^1 \big)
 = {\rm d} \mu_1 \wedge \w^2 + \mu_1 {\rm d} \w^2 - {\rm d}\mu_2 \wedge \w^1 - \mu_2 {\rm d}\w^1
 \\ \hphantom{* \Delta m }
 {}= \big( \mu_2 \w^2_1 + B_{11} \w^1 - m \w^1 + B_{12} \w^2\big) \wedge \w^2 - \mu_1 \w^2_i \wedge \w^i
 \\ \hphantom{* \Delta m= ( }
{} - \big( {-} \mu_1 \w^2_1 + B_{12} \w^1 + B_{22} \w^2 - m \w^2 \big) \wedge \w^1 + \mu_2 \w^1_i \wedge \w^i
\\ \hphantom{* \Delta m }
{} = ( B_{11} + B_{22} - 2 m) \w^1 \wedge \w^2.
\end{gather*}
Using $B_{11} + B_{22}=0$ from~\eqref{splitedseqns}, and taking Hodge star of the above, we~conclude that \mbox{$\Delta m = -2m$}.

Conversely, let $\Sigma \subset S^n$ be an arbitrary minimal surface which is carrying a moving frame $(\vy,\vv_1,\vv_2,\vv_3,\dots, \vv_n)$ where the unit vector $\vy$ represents position on the surface, $\vv_1$, $\vv_2$ are tangent to $\Sigma$, and the $\vv_r$ are tangent to $S^n$ but normal to the surface. To this moving frame we associate canonical forms $\eta^1$, $\eta^2$ and connection forms $\eta^1_2$ and $\eta^r_i$, such that
\begin{gather*}
{\rm d}\vy = \vv_i \eta^i, \qquad
 {\rm d}\vv_i = \vv_j \eta^j_i + \vv_r \eta^r_i
\end{gather*}
as $\R^{n+1}$-valued functions. Because $\Sigma$ is minimal, $\eta^r_i = H^r_{ij}\eta^j$ for some traceless $2\times 2$ matrices~$H^r$.

Define a mapping $\psi\colon \Sigma \times \R^+ \to \F$ (with $s$ as coordinate on the $\R^+$ factor) by
\begin{gather*}
\vx = s \vy,\qquad
\ve_0 = \vy, \qquad
\ve_1 = \vv_1, \qquad
\ve_2 = \vv_2, \qquad
\ve_r = \vv_r.
\end{gather*}
This gives a moving frame along the cone over $\Sigma$.

By differentiating $\vx$ and $\ve_0$, $\ve_1$, $\ve_2$, $\ve_r$ using~\eqref{deex} and~\eqref{dees}, we~compute that $\psi^* \w^i = s \eta^i$, $\psi^*\w^0 = {\rm d}s$, $\psi^* \w^i_0 = \eta^i$, $\psi^* \w^r_i = \eta^r_i$ and $\psi^* \w^r_0 = 0$. It~then follows that the components of the second fundamental form of the cone, relative to $\w^0$, $\w^1$, $\w^2$, are
\begin{gather*}
A^r = s^{-1}\begin{pmatrix} 0 & 0 \\ 0 & H^r \end{pmatrix} .
\end{gather*}
Let the components $m_i$, $m_{ij}$ of the covariant derivatives of $m$ on $\Sigma$ be defined by ${\rm d}m = m_i \eta^i$ and ${\rm d}m_i - m_j \eta^j_i = m_{ij} \eta^j$. Then taking $\mu = {\rm d}(ms)$ and using~\eqref{betadef} yields that the components of $\nabla\mu$ relative to $\w^0$, $\w^1$, $\w^2$ are
\begin{gather*}
B = s^{-1}\begin{pmatrix} 0 & 0 & 0 \\
0 & m_{11}+m & m_{12} \\ 0 & m_{12} & m_{22} + m \end{pmatrix} .
\end{gather*}
It is then easy to see that $\Delta m = - 2 m$ implies that~\eqref{C0expanded} is satisfied.
\end{proof}

\begin{Remark}
The minimality of $\Sigma \subset S^n$ is equivalent to its coordinates as a submanifold of $\R^{n+1}$ being eigenfunctions of $\Delta$ for eigenvalue $-2$. However, if we take $m$ to be one of these coordinate functions, the special Lagrangian submanifold in $T^* \R^{n+1}$ that results from the Borisenko construction is easily seen to be merely a translation of the conormal bundle of $M$.
\end{Remark}

\subsubsection[M is a twisted cone]
{$\boldsymbol M$ is a twisted cone} \label{twistconesec}

Recall from~\cite{Br} that the twisted cone over a minimal surface in the sphere is constructed as follows. Let $\vu\colon \Sigma \to S^n$ be a minimal immersion of a surface $\Sigma$, and let $f$ be a scalar function on $\Sigma$ satisfying $\Delta f = -2f$. (This, of course, is the same equation satisfied by the components of $\vu$ when $\vu$ is regarded as an $\R^{n+1}$-valued function.) Thus, the $\R^{n+1}$-valued $1$-form $\vbeta = \vu (*{\rm d}f) - f (*{\rm d}\vu)$ is closed, and the twisted cone is given by
\begin{gather} \label{twistyform}
\vX(s,t) = \vw(s) + t \vu(s), \qquad
s\in \Sigma, \quad t\in\R^+,
\end{gather}
where $\vw\colon \Sigma \to \R^{n+1}$ satisfies ${\rm d}\vw = \vbeta$. (It may be necessary to pass to the universal cover of $\Sigma$ for $\vw$ to be well-defined.)

Let $\calI$ be the Pfaffian system defined at the beginning of Section~\ref{splitsec} and let $\calI^+$ be a partial prolongation defined on $\F \times N \times \R^2$ by taking $p$, $q$ as coordinates on the last factor and adjoining $1$-forms
\begin{gather*}
\alpha_1 := \w^0_1 + p \w^1 + q\w^2 , \qquad \alpha_2 := \w^0_2 -q\w^1+ p\w^2.
\end{gather*}
These are analogous to the $1$-forms defined in~\eqref{coneadd78} but now we assume $q\ne 0$. Consequently, from Lemma~\ref{pqprolong}, it is necessary that $\tan\phi=0$. By differentiating the above expressions, we~can compute as in Section~\ref{splitconesec} that
\begin{gather} \label{twistedconetempeq}
{\rm d}p + \big(p^2 - q^2\big) \w^0 = h_1 \w^1 + h_2 \w^2, \qquad
{\rm d}q + 2pq \w^0 = h_2 \w^1 - h_1 \w^2,
\end{gather}
on any admissible integral manifold, for some undetermined functions $h_1$, $h_2$.

Let $\Mhat$ be an admissible integral manifold of $\calI^+$ for which $p$ and $q$ are both nonvanishing. If~its base $M$ is parametrized by an immersion~\eqref{twistyform}, and the form of the right-hand side indicates that the unit vector~$\vu$ must point along the rulings of $M$, and thus must coincide with vector $\ve_0$ of our adapted frame. In fact, one can check that if we choose $\vu= \ve_0$, $f=q/\big(q^2+p^2\big)$ and $t=p/\big(q^2+p^2\big)$, then there exists a $\vw$ such that $\vw + t \vu$ equals the position vector $\vx$ on $M$. Moreover, one can also compute explicitly using~\eqref{betadef} and~\eqref{twistedconetempeq} that, for these choices of $f$ and $t$, we~have
\begin{gather} \label{twistedconemuform}
\mu = m (*{\rm d}f) - f (*{\rm d}m) + {\rm d}(tm),
\end{gather}
where the slope $m$ is again an eigenfunction on $\Sigma$ satisfying $\Delta m = -2m$.

Conversely, we~have the following:
\begin{Theorem}
Let $\vu\colon\Sigma \to S^n$ be a minimal immersion and let $M$ be a twisted
cone over its image $\Sigma$, defined by data $(f,\vw)$. Let $m\colon\Sigma \to \R$ satisfy
$\Delta m =-2m$. Then $(M,\mu)$ is a~twi\-s\-ted-austere pair, where $\mu$ is given by~\eqref{twistedconemuform}, for $\tan\phi=0$.
\end{Theorem}

\begin{proof}[Sketch of proof]
Let $\F_S$ denote the orthonormal frame bundle of $S^n$, and let $\J$ be the Pfaffian system on $\F_S \times \R^{2(n-2) + n+9}$ that encodes the minimal surface condition for $\Sigma$, coupled with the equations satisfied by $f$, $m$ and $\vw$. (On the second factor in the product, we~use as coordinates the two free components of the second fundamental form in each normal direction, the three 1-jet variables for $f$, the five free 2-jet variables for $m$, and the components of $\vw$.) Let $\calI^{++}$ be the Pfaffian system on $\F \times \R^{2(n-2)+9}$ that is a further prolongation of $\calI^+$ including the free components of the derivatives of $p$, $q$ as additional variables. Both Pfaffian systems have rank $4n-1$, and one can define a map from the underlying manifold of $\calI^{++}$ to the underlying manifold of $\J$ such that $\calI^{++}$ is the pullback of $\J$, and there is a one-to-one correspondence between admissible integral surfaces of $\J$ and admissible integral 3-manifolds of $\calI^{++}$. Further details are left to the interested reader.
\end{proof}

\subsection{The non-split case}

In this subsection we assume that $(M,\mu)$ is a twisted-austere pair where the components of $\nabla \mu$ and $\II$ (relative to the adapted frame described at the beginning of Section~\ref{austeresec}) take the form~\eqref{austereABform} with $B_{10}^2 + B_{20}^2 > 0$ at every point.

\begin{Remark} \label{analytic-rmk}
In this section we are restricting to the open set $U$, where $B_{10}^2 + B_{20}^2 > 0$. Since~$M$ is austere, it is minimal, and hence real analytic. Moreover, in all cases in this section, the \mbox{$1$-form}~$\mu$ is also real analytic, as it is defined using solutions to a Laplace equation with real analytic right hand side. It~follows that $U$ must in fact be a \emph{dense} open set.
\end{Remark}

By rotating the frame vectors $\ve_1$, $\ve_2$ we may arrange that $B_{20}=0$ and $B_{10}>0$ at every point. It~then follows from~\eqref{C1minimal} and~\eqref{austereABform} that the diagonal entries of $A^r$ must all vanish, and hence~$|\II|$ is $1$-dimensional at each point. Therefore, we~may adapt the normal frame so that $A^r = 0$ for all $r>3$.

\begin{Proposition} \label{non-split-prelim}
Let $\big(M^3,\mu\big)$ be a twisted-austere pair where $M \subset \R^n$ is austere but is not a~generalized helicoid, and such that $\nabla \mu$ does not split. Then $M$ lies in $\R^4$.
\end{Proposition}
\begin{proof}
It is easy to check that the first prolongation of $|\II|$ has dimension zero, so this follows from Theorem~\ref{codim-reduction-thm} in Appendix~\ref{codim-reduction-sec}.
\end{proof}

For the rest of this section we will assume that $M$ lies in $\R^4$, and will continue to use indices $0, \dots, 3$ to label the members of the moving frame. We~will let $\calI$ denote the Pfaffian exterior differential system associated to our adapted frame, analogous to that defined at the beginning of Section~\ref{splitsec}, but with a much shorter list $\w^3$, $\beta_a$, $\Omega_a$ of $1$-form generators. Here, $\beta_a$ is as defined in~\eqref{betadef}, where $B$ is now assumed to have the form
\begin{gather*}
B=\begin{pmatrix} -\tan \phi & B_{10} & 0 \\ B_{10} & B_{11} & B_{12} \\ 0 & B_{12} & B_{22} \end{pmatrix}
\end{gather*}
subject to the condition~\eqref{C0expanded}, which now takes the form
\begin{gather} \label{nonsplitBcond}
B_{11}+B_{22} + B_{10}^2 \cos\phi (\sin\phi + \cos\phi B_{22})=0,
\end{gather}
and $\Omega_a = -\w^3_a + A_{ab} \w^b$, where
\begin{gather*}
A=\begin{pmatrix} 0 & 0 & 0 \\ 0 & 0 & h \\ 0 & h & 0 \end{pmatrix} .
\end{gather*}
The system $\calI$ thus has rank 7, and is defined on $\F \times N$ where $N \subset \R^4 \times \calS_3$ has coordinates $h$, $\mu_a$, $B_{ab}$ satisfying
$B_{00}=-\tan\phi$, $B_{20}=0$ and~\eqref{nonsplitBcond}. (We will solve this equation for $B_{11}$ in~terms of the other coordinates, and assume that $h$ and $B_{10}$ are positive. By Remark~\ref{analytic-rmk} this assumption will hold on a dense open set.)

\begin{Lemma} \label{non-split-cases}
The conclusions of Lemma~$\ref{pqprolong}$ apply in the non-split case as well, but solutions with $q\ne 0$ are not possible.
\end{Lemma}
\begin{proof}
The first assertion~\eqref{austpq} follows by the same argument made in the proof of Lemma~\ref{pqprolong} for the case where $|\II|$ is 1-dimensional,
and the correspondence between the values of $p$, $q$ and the branches of the Bryant classification is the same. To eliminate the possibility of twisted cones, we~compute the system 2-forms and, using the values given by~\eqref{austpq}, one can compute that
\begin{gather*}
\left.
\begin{aligned}
{\rm d}\beta_0 \wedge \w^1 &\equiv -B_{10} \big(2q\w^0 +\w^2_1\big) \wedge \w^1 \wedge \w^2 \\
{\rm d}\Omega_1 \wedge \w^2 &\equiv -h\big(q\w^0+2\w^2_1\big) \wedge \w^1 \wedge \w^2
\end{aligned}\right\} \mod \calI_1.
\end{gather*}
Linearly combining the 3-forms on the right-hand sides above to eliminate $\w^2_1 \wedge \w^1 \wedge \w^2$ shows that $B_{10} h q \w^1 \wedge \w^2 \wedge \w^3$ is in the ideal. Given our assumptions that $B_{10}$ and $h$ are positive, we~see that admissible integral manifolds with $q\ne 0$ are not possible.
\end{proof}

We now consider the two sub-cases given by Lemma~\ref{non-split-cases}. As before, $m = \ve_0 \intprod \mu = \mu_0$ is~the~slope.

\subsubsection[M is a cylinder]{$\boldsymbol M$ is a cylinder}

As in Section~\ref{splitcylindersec}, $M = \Sigma_0 \times\R$ where $\Sigma_0$ is minimal surface in $\R^3$, and we let $t$ denote the coordinate on the second factor, hence $\w^0 = {\rm d}t$. Using~\eqref{austpq} with $p=q=0$ shows that
\begin{gather} \label{nonsplitcylinderdm}
{\rm d}m \equiv -\tan \phi\, \w^0 + B_{10}\w^1 \quad \mod \calI_1.
\end{gather}
Therefore, we~see that now $m + t\tan \phi$ is non-constant, in contrast to Section~\ref{splitcylindersec}. Nevertheless,~$M$ is still described by the construction of Theorem~\ref{cylinderthm} for ambient space $\R^4$, but now it~follows from~\eqref{mu0eq} that the function $k=( m + t \tan \phi) \cos \phi$ is non-constant. Hence, not only is~$M$ a~product of the minimal surface $\Sigma_0 \subset \R^3$ with a line, it is also the family of parallel lines through a non-trivial minimal surface $\Sigma \subset \R^4$ which is a graph over $\Sigma_0$.

The set of such pairs $(\Sigma_0,\Sigma)$ forms a 5-parameter family (modulo rigid motions) and can be determined by solving a system of ordinary differential equations whose solutions are expressible using elliptic functions. Since equation~\eqref{nonsplitcylinderdm} shows that $m$ is constant along the asymptotic directions of $M$ that are annihilated by $\w^2$, the set of twisted-austere pairs $(M,\mu)$ where $M^3$ is an austere cylinder with $m+t\tan \phi$ non-constant is in one-to-one correspondence with the 4-parameter sub-family of these pairs $(\Sigma_0,\Sigma)$ where the minimal graph $\Sigma$ has constant height along one set of asymptotic lines of $\Sigma_0$.

\subsubsection[M is a cone]{$\boldsymbol M$ is a cone}

As in Section~\ref{splitconesec}, in this case we assume that $p > 0$ and $q=0$ identically, and we begin by~defining a partial prolongation. As before, we~use $p$ as a new coordinate and define $1$-forms
\begin{gather*} 
\alpha_1 := \w^0_1 + p \w^1, \qquad
\alpha_2 := \w^0_2 + p\w^2, \qquad
\alpha_3:={\rm d}p + p^2 \w^0.
\end{gather*}
We can compute that ${\rm d} \Omega_0\equiv 0$ and
\begin{align*}
{\rm d} \Omega_1 &\equiv -2h \w^2_1 \wedge \w^1 + \big({\rm d}h +hp \w^0\big) \wedge \w^2,\\
{\rm d} \Omega_2 &\equiv \big({\rm d}h+hp \w^0\big) \wedge \w^1 + 2h\w^2_1 \wedge \w^2
\end{align*}
modulo $\calI_1$, $\alpha_1$, $\alpha_2$. It~follows from the above equations that on any admissible integral manifold, there will be functions $u_1$, $u_2$ such that
\begin{gather*}
\w^2_1 = u_1 \w^1 + u_2 \w^2, \qquad
{\rm d}h = h\big({-}2u_2 \w^1 + 2u_1 \w^2 -p\w^0\big).
\end{gather*}
Because of this, we~will thus define the prolongation $\calI^+_1$ on $\F \times N \times \R^3$, with coordinates~$p$,~$u_1$,~$u_2$ on the last factor, by adjoining $\alpha_1$, $\alpha_2$, $\alpha_3$ as well as
\begin{gather*}
\alpha_4 := -\w^2_1 + u_1\w^1 + u_2 \w^2,
\\
\alpha_5 := -{\rm d}h + h\big({-} 2u_2 \w^1+2u_1 \w^2 -p\w^0\big).
\end{gather*}

\begin{Lemma} \label{coneconstant}
Austere bases of cone type with $\mu$ non-split only exist for $\tan\phi = 0$. $($As in Section~$\ref{splitconesec}$, this means that the slope must be constant along the rulings.$)$
\end{Lemma}
\begin{proof}
The computations are quite involved, but we carefully describe all the steps so that the reader will be able to fill in all the details if desired.

We will compute the 2-forms modulo the newly-added $1$-forms of $\calI^+$. Thus, all the congruences in this proof will be modulo $\calI^+_1$. To begin, one can compute that
\begin{gather*}
{\rm d}\beta_0 \equiv \big({\rm d}B_{10} +2B_{10} p \w^0 -B_{10} u_1 \w^2 \big) \wedge \w^1.
\end{gather*}

It follows that on any admissible integral manifold there is a smooth function $Z$ such that
\begin{gather} \label{deebZ}
{\rm d}B_{10} = -2B_{10} p \w^0+Z \w^1 +B_{10} u_1 \w^2.
\end{gather}
On the other hand, we~can compute that
\begin{gather*}
{\rm d}\beta_1 \equiv {\rm d}B_{10} \wedge \w^0 -\cos^2\phi \big( 2B_{10} (B_{22}+\tan\phi) {\rm d}B_{10} + \big(B_{10}^2 + \sec^2\phi\big) {\rm d}B_{22}\big) \wedge \w^1
\\ \hphantom{{\rm d}\beta_1 \equiv}
{} + {\rm d}B_{12}\wedge \w^2 + \cdots,
 \\
{\rm d}\beta_2 \equiv {\rm d}B_{12} \wedge \w^1 + {\rm d}B_{22}\wedge \w^2 + \cdots,
\end{gather*}
where for the moment we have omitted terms that are ``torsion'' (that is, linear combinations of~$\w^0 \wedge \w^1$, $\w^0 \wedge \w^2$ and $\w^1 \wedge \w^2$). These terms come into play when we linearly combine these 2-forms in $\calI^+$ so as to eliminate the terms involving ${\rm d}B_{12}$ and ${\rm d}B_{22}$, obtaining the following 3-form:
\begin{gather*}
{\rm d}\beta_1 \wedge \w^2 \!-\!\big(1\!+\!\cos^2\phi B_{10}^2\big) {\rm d}\beta_2 \wedge \w^1
\equiv -{\rm d}B_{10} \wedge \big( \w^2 \wedge \w^0\! +\! 2\cos^2\phi B_{10}(B_{22}\!+\!\tan\phi) \w^1 \wedge \w^2\big)
\\ \hphantom{{\rm d}\beta_1 \wedge \w^2 \!-\!\big(1\!+\!\cos^2\phi B_{10}^2\big) {\rm d}\beta_2 \wedge \w^1
\equiv\,}
{} + \big( 2p\tan\phi -B_{10} u_2 -\cos^2\phi B_{10}^3 u_2\big) \w^0 \wedge \w^1 \wedge \w^2.
\end{gather*}
Substituting in for ${\rm d}B_{10}$ from~\eqref{deebZ}, and solving for $Z$ so that the 3-form on the right vanishes, we~obtain
\begin{gather}
{\rm d}B_{10} = \big[ 2p\tan\phi - B_{10} u_2 + \cos^2\phi B_{10}^2 (4p(B_{22} + \tan\phi) - B_{10} u_2)\big] \w^1 \nonumber\\
\hphantom{{\rm d}B_{10} =}{} + B_{10} \big(u_1 \w^2 - 2 p \w^0\big). \label{deeb}
\end{gather}

Next, differentiating the right-hand side of~\eqref{deeb}, and using the value of ${\rm d}B_{10}$ given by~\eqref{deeb}, yields a 2-form $\Upsilon$ that must vanish on all integral submanifolds. Wedging this with $\w^2$ and linearly combining this with other 3-forms in $\calI^+$ yields
\begin{gather*}
\sec^2\phi \bigg(\Upsilon -\dfrac{B_{10}}{2h} {\rm d}\alpha_5\bigg)\wedge \w^2 + \big(4p B_{10}^2 {\rm d}\beta_2-B_{10}^3 {\rm d}\alpha_4\big) \wedge \w^1
\\ \qquad
{} \equiv 4p\big( 2B_{10}^3 u_2 -3B_{10}^2 p (B_{22}+\tan\phi) + p\tan\phi \sec^2\phi \big)\w^0 \wedge \w^1 \wedge \w^2.
\end{gather*}
Thus, all admissible integral submanifolds must lie in the zero locus of the polynomial
\begin{gather*}
S_1 := 2B_{10}^3 u_2 -3B_{10}^2 p (B_{22}+\tan\phi) + p\tan\phi \sec^2\phi.
\end{gather*}
We then differentiate the above expression and again use~\eqref{deeb} to compute
\begin{gather*}
{\rm d} S_1 \wedge \w^1 \wedge \w^2 + 3\sec^2\phi\bigg(\Upsilon -\dfrac{B_{10}}{2h} {\rm d}\alpha_5\bigg)\wedge \w^2 + 5B_{10}^3 {\rm d}\alpha_4 \wedge \w^1
\\ \qquad
{}\equiv 4\big({-}11 p B_{10}^3 u_2+9 B_{10}^2 p^2 (B_{22}+\tan\phi)+2 p^2 \tan\phi \sec^2\phi \big)\w^0 \wedge \w^1 \wedge \w^2.
\end{gather*}
This yields a second polynomial integrability condition, and eliminating $B_{22}$ between the two polynomials shows that all integral submanifolds must lie in the zero locus of
\begin{gather*}
S_2:= B_{10}^3 u_2 - p\tan\phi \sec^2\phi.
\end{gather*}
Differentiating the above expression and again using~\eqref{deeb}, we~obtain that
\begin{gather*}
{\rm d}S_2 \wedge \w^1 \wedge \w^2 +B_{10}^3 {\rm d}\alpha_4 \wedge \w^1 \equiv p\big({-}7B_{10}^3 u_2 + p\tan\phi \sec^2\phi\big)\w^0 \wedge \w^1 \wedge \w^2.
\end{gather*}
This last polynomial cannot vanish at the same time as $S_2$ unless $\tan\phi=0$.
\end{proof}

As in Section~\ref{splitconesec}, we~conclude that the slope $m$ is a well-defined function on the minimal surface $\Sigma$ inside $S^3$, and satisfies $\Delta m = -2m$. However, in this case $m$ and $\Sigma$ turn out \emph{not} to be arbitrary. This is because the computations in the proof of Lemma~\ref{coneconstant} imply that $u_2=0$ and $B_{22}=0$ identically on all such solutions. Consequently, the system $\calI^+$ simplifies. For example, equation~\eqref{deeb} along with the vanishing of $\alpha_3$, $\alpha_4$, $\alpha_5$ now implies that $B_{10}^2$ is a constant multiple of~$hp^3$.

After an additional prolongation step, we~obtain a Frobenius system. This means that solutions of this type are determined by solving systems of ODE, and so depend on finitely many constants. While leaving the details for the interested reader, the end result is that, up to a~rigid motion, the surface $\Sigma$ is the torus in $S^3$ that is parametrized by
\begin{gather*}
(t,u) \mapsto [\cos(t) \cos(au), \cos(t)\sin(au), \sin(t) \cos(u/a), \sin(t)\sin(u/a)],
\end{gather*}
where $a$ is a positive constant. (Notice that the surface is compact if $a^2$ is rational.) The slope function is given by
\begin{gather*}
m = (c_1 \cos(au)+ c_2 \sin(au)) \cos(t) + (c_3 \cos(u/a) +c_4 \sin(u/a)) \sin(t),
\end{gather*}
while on $M$, the $1$-form is $\mu = {\rm d}(ms) + c_5 {\rm d}u$. (Here, $c_1, \dots, c_5$ are arbitrary constants. However, since $m$ is a linear combination of the $\R^4$ coordinates of $\Sigma$, the constant $c_5$ should be chosen to be nonzero so that the resulting special Lagrangian submanifold is not just a translation of~$N^* M$.)

\section{Classification results} \label{sec:classification}

In this section, we~prove that the only examples of twisted-austere 3-folds in Euclidean space are precisely those that we have already discussed. More precisely we have the following result.

\begin{Theorem} \label{thm:class}
Let $(M,\mu)$ be a twisted-austere pair, where $M^3 \subset \R^{n+1}$ is not totally geodesic. Then either $M$ is austere or $M$ is a cylinder.
\end{Theorem}
\begin{proof}
As in the proof of Lemma~\ref{coneconstant}, we~describe all the steps and leave the details to the reader.

By Proposition~\ref{thm1step3} we can assume that $M$ falls into case $(i)$ of Proposition~\ref{thm1step2}, because otherwise $M$ is an austere generalized helicoid. Thus, as we did in Section~\ref{austeresec}, we~may adapt a~moving frame $(\ve_0, \dots, \ve_n)$ along $M$ so that $\nabla \mu$ and the second fundamental form in the direction of $\ve_r$ (for $3 \le r\le n)$ are represented respectively by
\begin{gather*} 
B= \begin{pmatrix} -\tan\phi & v_1 & v_2 \\ v_1 & B_{11} & B_{12} \\ v_2 & B_{12} & B_{22} \end{pmatrix} ,
\qquad
A^r = \begin{pmatrix} 0 & 0 & 0 \\ 0 & A^r_{11} & A^r_{12} \\ 0 & A^r_{12} & A^r_{22}
\end{pmatrix} .
\end{gather*}
In terms of these, the twisted-austere condition~\eqref{condc1} takes the form
\begin{gather} \label{C1general}
\vv^{\rm T} \adj(\ba^r) \vv + \sec^2\phi \tr(\ba^r) = 0,
\end{gather}
where $\vv^{\rm T} = \begin{bmatrix} v_1 & v_2 \end{bmatrix}$ and $\ba^r$ is the lower-right $2\times 2$ block of $A^r$. If $v_1=v_2=0$ identically then~\eqref{C1general} implies that $M$ is minimal, and hence austere since $\det(A^r)=0$ already. Thus, we~will assume from now on that $M$ is not minimal, and hence that one of $v_1$, $v_2$ is nonzero at each point.

Because equation~\eqref{C1general} is linear condition on $A^r$, we~see that $\dim |\II| \le 2$ at each point. First, we~will assume that $|\II|$ is 2-dimensional on an open set in $M$, and we will adapt the frame, by rotating the normal vectors, so that $A^r=0$ for $r>4$. (The case where $|\II|$ is 1-dimensional, including the case where $M\subset \R^4$, will be discussed later.) By Lemma~\ref{no-rank-one} we know that $A^r$ has rank 2 for some $r$. We~now further adapt the frame by rotating $\ve_1$, $\ve_2$ so that $v_2=0$ identically. With this adaptation, equation~\eqref{C1general} now reads
\begin{gather*}
v_1^2 A^r_{22} + \sec^2\phi \big(A^r_{11} + A^r_{22}\big)=0,\qquad r=3,4.
\end{gather*}
Thus the vectors $\big[A^3_{11}, A^3_{22}\big]$ and $\big[A^4_{11}, A^4_{22}\big]$ are linearly dependent, and we may rotate frame vectors $\ve_3$, $\ve_4$ so as to arrange that the diagonal entries of $A^4$ are zero.

We now consider the Pfaffian system for which this moving frame corresponds to an admissible integral manifold. As in Section~\ref{austeresec} we let $\F$ be the frame bundle of $\R^{n+1}$, and on $N=\R^{11}$ we take coordinates $\mu_a$, $v_1>0$, $B_{ij}$, $A^3_{ij}$ and $A^4_{12}$. (We will not yet impose conditions~\eqref{C1general} on the components of $A^3$, or impose~\eqref{C0expanded} on the components of $B$.) We define the $1$-forms $\beta_a$, $\theta^r$, and~$\Omega^r_a$ for $0\le a \le 2$ and $3 \le r \le n$ as in~\eqref{bigsysdef}, and let~$\I$ be the Pfaffian system on $\F \times N$ generated by $\vbeta$, $\vtheta$, and $\vOmega$.

As in the proof of Lemma~\ref{pqprolong}, we~compute that ${\rm d} \Omega^r_0 \equiv -A^r_{ij} \w^i_0 \wedge \w^j$ mod $\I_1$. Our assumption that $M$ is not minimal implies that $A^3$ has rank two, and thus by the Cartan Lemma on any admissible integral manifold there must be functions $P_{ij} =P_{ji}$ such that $\w^i_0 = P_{ij} A^3_{jk} \w^k$. (Note that these coefficients are \emph{different} from the $P^i_j$ introduced in the proof of Lemma~\ref{pqprolong}.) Using this, we~can compute that
\begin{gather*}
{\rm d} \Omega^4_0 \equiv A^4_{12} \big(A^3_{22} P_{22} - A^3_{11} P_{11}\big)\w^1 \wedge \w^2,
\end{gather*}
and thus there must be a function $p$ such that $P_{11}=p A^3_{22}$ and $P_{22} = p A^3_{11}$.

As in Sections~\ref{splitconesec} and~\ref{twistconesec}, we~construct a partial prolongation $\I^+$ by adjoining the $1$-forms
\begin{gather*}
\alpha_i := \w^0_i + P_{ij} A^3_{jk} \w^k, \qquad\text{with} \quad P_{11}=p A^3_{22},\quad P_{22} = p A^3_{11}
\end{gather*}
defined on $\F\times N \times \R^2$, with $p$ and $P_{12}$ as additional variables. Computing the system 2-forms then uncovers the additional integrability condition $P_{12}=-p A^3_{12}$. Restricting the system $\I^+$ to~the submanifold where this condition holds, one computes that
\begin{gather*}
{\rm d}\ve_0 \equiv p \big(A^3_{11} A^3_{22} - \big(A^3_{12}\big)^2\big) \big(\ve_1 \w^1 + \ve_2 \w^2\big)\mod \I^+,
\end{gather*}
indicating that $M$ must be a cone over a surface in $S^n$ if $p\ne0$, or a cylinder if $p=0$ identically.

Imposing the twisted-austere condition~\eqref{C1general} amounts to restricting to the smooth submanifold where the polynomial $Q_0:= v_1^2 A^3_{22} + \sec^2\phi\big(A^3_{11} + A^3_{22}\big)$ vanishes. We can compute that
\begin{gather*}
{\rm d}Q_0 \equiv -p \big(A^3_{11} A^3_{22} - \big(A^3_{12}\big)^2\big) \big(Q_0 + 4 v_1^2 A^3_{22}\big)\w^0 \mod \I^+,\, \w^1, \, \w^2.
\end{gather*}
Thus, admissible integral submanifolds lying within this locus must also have $p=0$ (in which case $M$ is a cylinder) or $v_1 A^3_{22}=0$ (in which case $Q_0=0$ implies that $M$ is minimal).

Now consider the case where $\dim |\II|=1$ at each point. We~will not initially assume that $v_2=0$. We again define the partial prolongation $\I^+$ by adjoining $1$-forms $\alpha_i := \w^0_i + P_{ij} A^3_{jk} \w^k$, but now we cannot assume any relations among the $P_{ij}$ other than $P_{ij}=P_{ji}$. Again, to impose~\eqref{C1general} we must restrict to the zero locus of
\begin{gather*}
Q_0:=v_1^2 A^3_{22} + v_2^2 A^3_{11} - 2 v_1 v_2 A^3_{12} + \sec^2\phi \big(A^3_{11} + A^3_{22}\big).
\end{gather*}
By computing ${\rm d} Q_k \equiv Q_{k+1} \w^0 \mod \I^+$, $\w^1$, $\w^2$, we~obtain additional polynomials $Q_1$, $Q_2$, $Q_3$ in whose zero locus any admissible integral manifold must lie. By rotating the frame we can arrange that $v_2=0$. This simplifies the polynomials, and we find that we must have $P_{11}=0$ or~$\det(P_{ij})=0$ on the common zero locus.

Setting $P_{11}=0$ and $v_2=0$ in $Q_1$ implies that $P_{12}=0$, and then substituting these in $Q_2$ implies that $P_{22}=0$, and thus in this case $M$ is a cylinder. Therefore if $M$ is not a cylinder, we must have $\det(P_{ij})=0$. This condition is invariant under rotating the vectors $\ve_1$, $\ve_2$, so in this case we may arrange that $P_{12}=P_{22}=0$ instead of $v_2=0$. Substituting in $Q_2$ yields $A^3_{11}=0$, and substituting these values into $Q_1$ gives either $A^3_{12}=0$ or $3 v_1^2=\sec^2\phi$. If $A^3_{12}=0$ then substituting into $Q_0$ yields that $M$ is totally geodesic. In the remaining case we have $P_{12}=P_{22}=A^3_{11}=0$ and $v_1$ has a constant value. Computing the prolongation of $\I^+$ in this case yields additional integrability conditions that imply $M$ must be totally geodesic.
\end{proof}

\begin{Remark}
In light of Proposition~\ref{thm1step3} and the classification theorem just proved, the only remaining possibility for twisted-austere 3-folds other than those discussed in Sections~\ref{cylindersec} and~\ref{austeresec} is that $M$ is a generalized helicoid in $\R^5$. However, in this case the only possible values for the $1$-form $\mu$ produce, via the Borisenko construction, a special Lagrangian submanifold in $\R^{10}$ that is a translation of the conormal bundle of~$M$ by a constant vector.
\end{Remark}

\appendix

\section{Appendix}

In this appendix, we~collect some linear algebraic results that are needed in the main body of~the paper. These include two identities relating the elementary symmetric polynomials with the operation of taking the adjugate matrix, as well as some results on the spans of singular symmetric matrices.

\subsection[Identities relating sigma j and adj]
{Identities relating $\boldsymbol{\sigma_j}$ and $\boldsymbol{\adj}$} \label{sec:sigma-adj}
In this section we prove the two fundamental identities~\eqref{expandS1} and~\eqref{expandS2} that are crucially used in~our classification. We~prove a more general result than~\eqref{expandS2} valid for any $k$, whereas for~\eqref{expandS1} we restrict to the case $k=3$.

We first recall some basic facts about the elementary symmetric polynomials and the adjugate matrix, to fix notation. Let $A$ be a $k \times k$ matrix with complex entries. For $j=0, \dots, k$ we~define the $j^{\text{th}}$ elementary symmetric polynomial $\sigma_j(A)$ of $A$ by the expression
\begin{gather} \label{eq:sigmadefn}
\det (I + t A) = \sum_{j=0}^k t^j \sigma_j (A).
\end{gather}
It is clear from~\eqref{eq:sigmadefn} that $\sigma_j \big(P^{-1} A P\big) = \sigma_j (A)$ for all $j$. In particular we have $\sigma_0 (A) = 1$, $\sigma_1(A) = \tr A$, and $\sigma_k (A) = \det A$. Moreover, each $\sigma_j$ is a homogeneous polynomial of degree $j$ in the entries of~$A$, so $\sigma_j (\lambda A) = \lambda^j \sigma_j (A)$ for all $\lambda \in \C$.

Suppose that $A$ is invertible. Then we can compute
\begin{gather*}
\det\big(I + t A^{-1}\big) = \det \big(t A^{-1} \big(I + t^{-1} A\big) \big) = t^k (\det A)^{-1} \sum_{j=0}^k \big(t^{-1}\big)^j \sigma_j (A)
\\ \hphantom{\det\big(I + t A^{-1}\big)}
 {}= \frac{1}{(\det A)} \sum_{j=0}^k t^{k-j} \sigma_j (A) = \frac{1}{(\det A)} \sum_{j=0}^k t^j \sigma_{k-j}(A).
\end{gather*}
We deduce from the above and~\eqref{eq:sigmadefn} that
\begin{gather} \label{eq:sigmaAinverse}
\sigma_j \big(A^{-1}\big) = \frac{1}{\det A} \sigma_{k-j} (A).
\end{gather}

The \emph{adjugate matrix} $\adj A$ is the unique $k \times k$ matrix satisfying
\begin{gather*} 
(\adj A) A = A (\adj A) = (\det A)I.
\end{gather*}
It is clear that $\adj \big(P^{-1} A P\big) = P^{-1} (\adj A) P$. Moreover, $\adj A$ is homogeneous of order $k-1$, so~$\adj(\lambda A) = \lambda^{k-1} \adj A$ for all $\lambda \in \C$.

If $A$ is invertible then $\adj A = (\det A) A^{-1}$. We~can use~\eqref{eq:sigmaAinverse} and the homogeneity of $\sigma_j$ to compute that $\sigma_j(\adj A) = \sigma_j \big( (\det A) A^{-1} \big) = (\det A)^j \sigma_j \big(A^{-1}\big) = (\det A)^{j-1} \sigma_{k-j}(A)$. By the density of invertible matrices we conclude that
\begin{gather} \label{eq:sigmaadjA}
\sigma_j (\adj A) = (\det A)^{j-1} \sigma_{k-j} (A) \qquad \text{for all $A$.}
\end{gather}
Note that the above is well-defined for all $A$ even when $j=0$, in which case it just says $1=1$.

\begin{Lemma} \label{lemma:sigma-adj-1}
Let $A$ and $C$ be $k \times k$ complex matrices with $C$ invertible. Then we have
\begin{gather*} 
\sigma_j \big(A C^{-1}\big) = (\det A)^{j+1-k} (\det C)^{-1} \sigma_{k-j} (C \adj A).
\end{gather*}
\end{Lemma}
\begin{proof}Assume first that $A$ is invertible. Using~\eqref{eq:sigmaAinverse}, we~compute
\begin{align*}
\sigma_j \big(A C^{-1}\big) & = \sigma_j \big( \big(C A^{-1}\big)^{-1} \big) = \frac{1}{\det \big(C A^{-1}\big)} \sigma_{k-j} \big(C A^{-1}\big) \\
& = \frac{\det A}{\det C} \sigma_{k-j} \big( (\det A)^{-1} C \adj A\big) = \frac{\det A}{\det C} (\det A)^{-(k-j)} \sigma_{k-j} (C \adj A) \\
& = (\det A)^{j+1-k} (\det C)^{-1} \sigma_{k-j} (C \adj A),
\end{align*}
as claimed. The result follows for all matrices $A$ by the density of invertible matrices.
\end{proof}

\begin{Proposition} \label{prop:sigma-adj-1}
Let $A$ and $B$ be $k \times k$ \emph{real} matrices, and assume that $B$ is symmetric. Let $C = I + {\rm i} B$. Then $C$ is invertible and we have
\begin{gather} \label{eq:sigma-adj-1}
\sigma_{k-1} \big(AC^{-1}\big) = \frac{\sigma_{k-1} (A) + {\rm i} \sigma_1 (B \adj A)}{\det C}.
\end{gather}
\end{Proposition}
\begin{proof}
The result~\eqref{eq:sigma-adj-1} we seek to prove is similarity invariant, so we can assume by the spectral theorem that $B$ is diagonal with real eigenvalues $\lambda_1, \dots, \lambda_k$. But then $C = I + {\rm i} B$ is diagonal with nonzero eigenvalues $1 + {\rm i} \lambda_1, \dots, 1 + {\rm i} \lambda_k$, and hence invertible. Applying Lemma~\ref{lemma:sigma-adj-1} with $j = k-1$ gives
\begin{gather*}
\sigma_{k-1} \big(A C^{-1}\big) = (\det C)^{-1} \sigma_1 (C \adj A).
\end{gather*}
But $\sigma_1 = \tr$ is linear, so $\sigma_1 (C \adj A) = \sigma_1 ( (I + {\rm i} B) \adj A) = \sigma_1 (\adj A) + {\rm i} \sigma_1 (B \adj A)$. The~proof is completed upon noting that $\sigma_1 (\adj A) = \sigma_{k-1} (A)$ from~\eqref{eq:sigmaadjA}.
\end{proof}

The next result is used to establish the second fundamental identity of this section.
\begin{Lemma} \label{lemma:sigma-adj-2}
Let $B$ be a symmetric $3 \times 3$ real matrix and let $z \in \C$. Then we have
\begin{gather} \label{eq:sigma-adj-lemma2}
\adj(I + z B) = I + z( \sigma_1(B) I - B) + z^2 \adj B.
\end{gather}
\end{Lemma}
\begin{proof}
By similarity invariance, we~can assume that $B$ is diagonal with real entries. We~compute explicitly that
\begin{gather*}
B = \begin{pmatrix} \lambda & 0 & 0 \\ 0 & \mu & 0 \\ 0 & 0 & \nu \end{pmatrix} , \qquad \adj B = \begin{pmatrix} \mu \nu & 0 & 0 \\ 0 & \lambda \nu & 0 \\ 0 & 0 & \lambda \mu \end{pmatrix} ,
\\[1ex]
I + z B = \begin{pmatrix} 1 + z \lambda & 0 & 0 \\ 0 & 1 + z \mu & 0 \\ 0 & 0 & 1 + z \nu \end{pmatrix} ,
\\[1ex]
\adj(I + z B) = \begin{pmatrix} (1 + z \mu) (1 + z \nu) & 0 & 0 \\ 0 & (1 + z \lambda) (1 + z \nu) & 0 \\ 0 & 0 & (1 + z \lambda) (1 + z \mu) \end{pmatrix} .
\end{gather*}
Then~\eqref{eq:sigma-adj-lemma2} can be directly verified. For example, the $(1,1)$ entry gives
\begin{gather*}
 (1 + z \mu) (1 + z \nu) = 1 + z ( (\lambda + \mu + \nu) - \lambda) + z^2 \mu \nu,
 \end{gather*}
which is clearly true.
\end{proof}

Before we can state the final result of this section, we~need to introduce some more notation. It~is well-known that
\begin{gather*} 
\sigma_2 (A) = \frac{1}{2} ( \sigma_1 (A) )^2 - \frac{1}{2} \sigma_1 \big(A^2\big).
\end{gather*}
This is the simplest of \emph{Newton's identities}. It~can be verified directly for a diagonal matrix, which implies the general case because the diagonalizable matrices are dense. By homogeneity, $\sigma_2$ is a~quadratic form on the space of matrices, and by polarization we obtain an induced symmetric bilinear form, which we denote by $\{ \cdot, \cdot \}$. Explicitly,
\begin{gather} \label{eq:sigma2-metric}
2 \{ A, B \} = \sigma_1 (A) \sigma_1 (B) - \sigma_1 (AB).
\end{gather}

\begin{Remark} \label{rmk:sigma2-metric}
The positive-definite Frobenius norm $\langle \cdot, \cdot \rangle$ on real matrices is given by $\langle A, B \rangle = \sigma_1 \big(A^{\rm T} B\big) = \tr \big(A^{\rm T} B\big) = \sum_{i,j} A_{ij} B_{ij}$. Note that $\tr A = \langle A, I \rangle$. Thus the traceless symmetric matrices are orthogonal to the identity matrix $I$ with respect to $\langle \cdot, \cdot \rangle$. Let $A$, $B$ be symmetric. We~can write $A = \frac{1}{k} (\tr A) I + A_0$ where $A_0$ is traceless and similarly for $B$. Then using~\eqref{eq:sigma2-metric} we~have
\begin{align*}
2 \{ A, B \} & = (\tr A) (\tr B) - \langle A, B \rangle = (\tr A) (\tr B) - \frac{1}{k^2} (\tr A) (\tr B) \langle I, I \rangle - \langle A_0, B_0 \rangle \\
& = \frac{k-1}{k} (\tr A) (\tr B) - \langle A_0, B_0 \rangle.
\end{align*}
The above computation shows that for $k > 1$ the symmetric bilinear form $\{ \cdot, \cdot \}$ is a \emph{Lorentzian} inner product on the space of symmetric $k \times k$ real matrices, that is, with signature $\big(1, \frac{k(k+1)}{2} - 1\big)$. This fact is used several times in this paper.
\end{Remark}

\begin{Proposition} \label{prop:sigma-adj-2}
Let $A$ and $B$ be $3 \times 3$ \emph{real} matrices, and assume that $A$ is invertible and that~$B$ is symmetric. Let $C = I + {\rm i} B$. Then $C$ is invertible and we have
\begin{gather*} 
\sigma_1 \big(AC^{-1}\big) = \frac{\sigma_1(A(I - \adj B)) + 2{\rm i} \{ A, B \} }{\det C}.
\end{gather*}
\end{Proposition}

\begin{proof}
The invertibility of $C$ was proved in Proposition~\ref{prop:sigma-adj-1}. Applying Lemma~\ref{lemma:sigma-adj-2} with $z={\rm i}$ gives
\begin{gather*}
 \adj C = \adj(I + {\rm i} B) = I + {\rm i} \sigma_1 (B) I - {\rm i}B - \adj B.
 \end{gather*}
Multiplying both sides on the left by $A$ and writing $\adj C = (\det C) C^{-1}$ gives
\begin{gather*}
 (\det C) A C^{-1} = A + {\rm i} \sigma_1 (B) A - {\rm i} AB - A \adj B.
 \end{gather*}
Taking $\sigma_1 = \tr$ of both sides and using linearity gives
\begin{gather*}
(\det C) \sigma_1 \big(A C^{-1}\big) = \sigma_1 (A - A \adj B) + {\rm i} ( \sigma_1 (A) \sigma_1 (B) - \sigma_1 (AB) ).
\end{gather*}
Using~\eqref{eq:sigma2-metric} completes the proof.
\end{proof}

\subsection{Spans of singular symmetric matrices} \label{sec:app-singular}

Let $\calS_n$ denote the space of $n\times n$ symmetric matrices with real entries, and let $\D_n \subset \calS_n$ be the affine variety of symmetric matrices with vanishing determinant. We~determine the maximal linear subspaces of $\D_n$ up to $O(n)$-conjugation, for $n=2$ and $n=3$.

\begin{Proposition} \label{sing2max}
Let $\W \subset \D_2$ be a maximal linear subspace. Then $\dim \W =1$ and is $O(2)$-con\-jugate to the span of $\left(\begin{smallmatrix} 1 &0 \\ 0 & 0 \end{smallmatrix}\right)$.
\end{Proposition}
\looseness=1\begin{proof} On the space $\calS_2$, the determinant is a quadratic form with signature $(1,2)$, and thus~$\D_2$ contains no linear subspaces of dimension greater than one. The~result follows by diagona\-lization.
\end{proof}

\begin{Proposition} \label{sing3max}
Let $\W\subset \D_3$ be a maximal linear subspace. Then $\W$ is $3$-dimensional and is $O(3)$-conjugate to one of
\begin{gather*}
 \W_1 = \left\{\!\! \begin{pmatrix} * & * & 0 \\ * & * & 0 \\ 0 & 0 & 0\end{pmatrix}\!\! \right\} ,\qquad
\W_2 = \left\{\!\! \begin{pmatrix} * & * & * \\ * & 0 & 0 \\ * & 0 & 0\end{pmatrix}\!\! \right\} .
\end{gather*}
\end{Proposition}
\begin{proof}
Let $\V \subset \D_3$ be an arbitrary linear subspace. If $\dim \V=1$ then by diagonalization $\V$ is conjugate to a subspace of $\W_1$. Thus we can assume $\dim \V \ge 2$.

\emph{Case one:} Suppose $\V$ contains a rank one matrix $A_0$. By $O(3)$-conjugation we can assume that
\begin{gather*} 
A_0 = \begin{pmatrix} 1 & 0 & 0 \\ 0 & 0 & 0 \\ 0 & 0 & 0 \end{pmatrix} .
\end{gather*}
Let $\V' = \{ B \in \V \,|\, B_{11}=0\}$. Then for any $B \in \V'$, by expansion along the top row, we have
\begin{gather} \label{zeroline1}
0 = \det(B+ t A_0) = \big(B_{22}B_{33}-B_{23}^2\big) t + \det B \qquad
\forall\, t \in \R.
\end{gather}
Let $p\colon \V' \to \calS_2$ denote the linear projection that gives the lower-right $2\times 2$ block. Then from the vanishing of the leading coefficient in~\eqref{zeroline1}, we~see that $p(\V') \subset \D_2$. By Proposition~\ref{sing2max}, we~know $\dim p(\V') \leq 1$. If $p(\V') = \{ 0 \}$ then $\V \subseteq \W_2$. If not then we can assume by conjugation that $B_{23}=B_{33}=0$ for all $B\in \V'$. Now equation~\eqref{zeroline1} reads $0 = B_{22} B_{13}^2$. If $B_{22}=0$ for all $B \in \V'$ then $\V = \{ A_0 \}_\R \oplus \V' \subset \W_2$. If not, then $\V'$ contains a matrix with $B_{13} = 0$ and hence by scaling a matrix $B_0$ of the form
\begin{gather*}
B_0 = \begin{pmatrix} 0 & \lambda & 0 \\ \lambda & 1 & 0 \\ 0 & 0 & 0 \end{pmatrix} ,\qquad \lambda \in \R.
\end{gather*}
In this case, let $\V'' = \ker p = \{ B \in \V' \,|\, B_{22}=0 \}.$ Then for any $B \in \V''$ we can compute that
\begin{gather*}
0 = \det (B + t B_0) = - B_{13}^2 t, \qquad\forall\, t \in \R.
\end{gather*}
Thus, $B_{13} = 0$ for all $B \in \V''$, so $\V'' \subset \W_1$ and $\V = \{ A_0 \}_\R \oplus \{ B_0 \}_\R \oplus \V'' \subset \W_1$.

\emph{Case two:} Suppose that $\V$ contains no rank one matrices. Thus every nonzero matrix in~$\V$ has exactly two nonzero eigenvalues and one zero eigenvalue. Hence, the space $\V$ does not intersect the cone defined by the equation $\sigma_2(A)=0$ except at the origin in $\calS_3$. Since $\sigma_2$ is a~quadratic form on $\calS_3$ with signature $(1,5)$, and $\V$ has dimension at least two, the restriction of~$\sigma_2$ to $\V$ will be negative on an open subset in $\V$. If $\sigma_2 > 0$ anywhere in $\V$, then by continuity there would be a nonzero $A \in \V$ such that $\sigma_2(A) = 0$. But we have ruled that out in this case, so we conclude that $\sigma_2(A) < 0$ for all nonzero $A \in \V$.

Now fix a rank two matrix $A_1$ in $\V$. Since $\sigma_2 (A_1) < 0$, it has one positive and one negative eigenvalue. Thus we can assume using $O(3)$-conjugation that $A_1$ takes the form
\begin{gather*}
A_1 = \begin{pmatrix} 1 & 0 & 0 \\ 0 & -\lambda^2 & 0 \\ 0 & 0 & 0 \end{pmatrix} , \qquad \lambda > 0.
\end{gather*}
Let $\V' = \{ B \in \V \,|\, B_{11}=0\}$. Then for $B \in \V'$ we must have
\begin{gather} \label{eq:sing2max-temp}
0 = \det(B + t A_1) = -\lambda^2 B_{33}t^2 + \big(B_{22}B_{33} - B_{23}^2 + \lambda^2 B_{13}^2\big)t + \det(B), \qquad \forall\, t \in \R.
\end{gather}
Thus, $B_{33}=0$ and $\V'$ either lies in the subspace where $B_{23}=\lambda B_{13}$ or in the subspace where $B_{23}=-\lambda B_{13}$; without loss of generality, we~can assume the former.

Under the assumption $B_{23}=\lambda B_{13}$, we have $\det(B) = B_{13}^2 (2\lambda B_{12}-B_{22})$. The function that takes a matrix in $\V'$ to its $(1,3)$-entry is a linear functional on the vector space $\V'$, so its kernel is either all of $\V'$ or a subspace of $\V'$ of positive codimension. Hence, either $B_{13} = 0$ for all matrices $B$ in $\V'$, or else $B_{13}$ is nonzero on a dense open subset of $\V'$. If $B_{13}=0$ for all matrices in $\V'$, then $\V' \subset \W_1$ and hence $\V \subset \W_1$. Otherwise, $B_{13} \ne 0$ on a dense open subset of $\V'$, and then $\det(B)=0$ from~\eqref{eq:sing2max-temp} implies that $B_{22}=2\lambda B_{12}$ for all matrices in $\V'$. In that case, conjugating the matrices in $\V'$ by the rotation matrix
\begin{gather*}
R = \begin{pmatrix} \cos \theta & \sin \theta & 0 \\ -\sin\theta & \cos \theta & 0 \\ 0 & 0 & 1\end{pmatrix} ,
\end{gather*}
where $\sin\theta=\lambda \cos \theta$, yields a matrix in $\W_2$.  Since $R A_1 R^{-1} \in \W_2$ as well, we conclude that $\V$ is conjugate to a subspace of $\W_2$.
\end{proof}

\subsection{A result for codimension reduction} \label{codim-reduction-sec}

In this section we establish a technical result that is used in the proofs of Propositions~\ref{thm1step3} and~\ref{non-split-prelim}.

Let $V$ and $W$ be real vector spaces. Given a linear subspace $L \subset S^k V^* \otimes W$, the \emph{prolongation} of $L$ is defined to be
\begin{gather*}
L^{(1)} = V^* \otimes L \cap S^{k+1} V^* \otimes W.
\end{gather*}
This definition arises in the study of the tableaux associated to systems of linear first-order PDE. See~\cite[Chapter VIII]{BCG3} for more details. For example, one can check that if $V=\R^n$, $W=\R$ and~$L$ is the set of symmetric $n\times n$ matrices in block form{\samepage
\begin{gather*}
\begin{pmatrix} 0 & B \\ B^t & 0 \end{pmatrix},
\end{gather*}
where $B$ is an arbitrary $k \times (n-k)$ matrix, then $L^{(1)} = 0$.}

We use the above definition (in the special case where $W=\R$) to formulate a codimension-reduction theorem for submanifolds of Euclidean space.

\begin{Theorem} \label{codim-reduction-thm}
Let $M^m \subset \R^N$ be a smooth connected submanifold with second fundamental form $\II$ such that the first normal bundle $N^1 M$ has constant rank~$\rho$. $($Recall that the fiber at~$p$ of~$N^1 M$ is the image of $\II_p\colon T_p M \otimes T_p M \to N_p M.)$ If at each point $p$ in $M$, the set~$|\II|_p$ as a~subspace of $S^2 T_p^* M$ satisfies $|\II|_p^{(1)}=0$, then $M$ is contained in a totally geodesic submanifold~$R$ of~dimension $m+\rho$ which is tangent to $T_p M \oplus N^1_p M$ at each $p \in M$.
\end{Theorem}
\begin{proof}
Near any point of $M$, choose an orthonormal frame $\ve_1, \dots, \ve_N$ such that $\ve_1, \dots, \ve_m$ span $T_p M$ and $\ve_{m+1}, \dots, \ve_{m+\rho}$ span $N^1_p M$. (In what follows, use index ranges $1 \le \alpha, \beta \le N$, $1\le i,j,k \le m$, $m< a,b \le m+\rho$ and $r,s > m+\rho$.) Let $\omega^i$, $\omega^\alpha_\beta$ be the canonical and connection $1$-forms associated to this moving frame along $M$. Then
\begin{gather} \label{defh}
{\rm d}\ve_j \equiv e_a h^a_{jk} \omega^k \mod T_p M,
\end{gather}
where $|\II|_p$ equals the span of the symmetric matrices $h^a_{jk}$. Suppose that
\begin{gather*}
{\rm d}\ve_a = e_r q^r_{ai} \omega^i	\mod T_pM, N^1_p M.
\end{gather*}
Then differentiating~\eqref{defh} shows that
\begin{gather*}
0 = e_r q^r_{ai} \omega^i \wedge h^a_{jk} \omega^k.
\end{gather*}
For each $r$, it follows that $S^r_{ijk} = q^r_{ai} h^a_{jk}$ satisfies $S^r_{ijk} = S^r_{ikj}$, and hence belongs to the space~$|\II|^{(1)}$, which is zero. Since the matrices $h^a_{jk}$ are linearly independent, the $q^r_{ai}$ vanish. Hence the span of $\{ \ve_i, \ve_a\}$ is fixed. If we let $R$ be the totally geodesic submanifold tangent to~this space at one point $p \in M$, then connecting any other point $q\in M$ to $p$ with a smooth curve in $M$ shows that all other points of $M$ must lie inside $R$.
\end{proof}

A generalization of this result to submanifolds in a Riemannian manifold may be found in~\cite[Section~4.2]{BY}.

\subsection*{Acknowledgements}
The authors thank the anonymous referees for useful feedback and comments that improved the quality of the paper.

\pdfbookmark[1]{References}{ref}
\LastPageEnding


\begin{thebibliography}{99}
\footnotesize\itemsep=0pt

\bibitem{Bo}
Borisenko A., Ruled special {L}agrangian surfaces, in Minimal Surfaces,
 \textit{Adv. Soviet Math.}, Vol.~15, Amer. Math. Soc., Providence, RI, 1993,
 269--285.

\bibitem{Br}
Bryant R.L., Some remarks on the geometry of austere manifolds, \href{https://doi.org/10.1007/BF01237361}{\textit{Bol.
 Soc. Brasil. Mat. (N.S.)}} \textbf{21} (1991), 133--157.

\bibitem{BCG3}
Bryant R.L., Chern S.S., Gardner R.B., Goldschmidt H.L., Griffiths P.A.,
 Exterior differential systems, \textit{Mathematical Sciences Research
 Institute Publications}, Vol.~18, \href{https://doi.org/10.1007/978-1-4613-9714-4}{Springer-Verlag}, New York, 1991.

\bibitem{BY}
Chen B.-Y., Riemannian submanifolds, in Handbook of Differential Geometry,
 {V}ol.~{I}, \href{https://doi.org/10.1016/S1874-5741(00)80006-0}{North-Holland}, Amsterdam, 2000, 187--418.

\bibitem{HL}
Harvey R., Lawson Jr. H.B., Calibrated geometries, \href{https://doi.org/10.1007/BF02392726}{\textit{Acta Math.}}
 \textbf{148} (1982), 47--157.

\bibitem{II}
Ionel M., Ivey T., Austere submanifolds of dimension four: examples and maximal
 types, \href{https://doi.org/10.1215/ijm/1318598678}{\textit{Illinois~J. Math.}} \textbf{54} (2010), 713--746,
 \href{https://arxiv.org/abs/0906.4477}{arXiv:0906.4477}.

\bibitem{II2}
Ionel M., Ivey T., Ruled austere submanifolds of dimension four,
 \href{https://doi.org/10.1016/j.difgeo.2012.07.007}{\textit{Differential Geom. Appl.}} \textbf{30} (2012), 588--603,
 \href{https://arxiv.org/abs/1011.4961}{arXiv:1011.4961}.

\bibitem{IKM}
Ionel M., Karigiannis S., Min-Oo M., Bundle constructions of calibrated
 submanifolds in {${\mathbb R}^7$} and {${\mathbb R}^8$}, \href{https://doi.org/10.4310/MRL.2005.v12.n4.a5}{\textit{Math. Res.
 Lett.}} \textbf{12} (2005), 493--512, \href{https://arxiv.org/abs/math.DG/0408005}{arXiv:math.DG/0408005}.

\bibitem{J-ruled}
Joyce D., Ruled special {L}agrangian 3-folds in {$\mathbb C^3$}, \href{https://doi.org/10.1112/S0024611502013485}{\textit{Proc.
 London Math. Soc.}} \textbf{85} (2002), 233--256, \href{https://arxiv.org/abs/math.DG/0012060}{arXiv:math.DG/0012060}.

\bibitem{J-book}
Joyce D., Riemannian holonomy groups and calibrated geometry, \textit{Oxford
 Graduate Texts in Mathematics}, Vol.~12, Oxford University Press, Oxford,
 2007.

\bibitem{KL}
Karigiannis S., Leung N.C.-H., Deformations of calibrated subbundles of
 {E}uclidean spaces via twisting by special sections, \href{https://doi.org/10.1007/s10455-012-9317-1}{\textit{Ann. Global
 Anal. Geom.}} \textbf{42} (2012), 371--389, \href{https://arxiv.org/abs/1108.6090}{arXiv:1108.6090}.

\bibitem{SYZ}
Strominger A., Yau S.-T., Zaslow E., Mirror symmetry is {$T$}-duality,
 \href{https://doi.org/10.1016/0550-3213(96)00434-8}{\textit{Nuclear Phys.~B}} \textbf{479} (1996), 243--259,
 \href{https://arxiv.org/abs/hep-th/9606040}{arXiv:hep-th/9606040}.

\end{thebibliography}
\end{document}